\documentclass{amsart}
\usepackage{ifthen}
\usepackage{amsmath}
\usepackage{amssymb}
\usepackage{amsfonts}
\usepackage{amsthm}
\usepackage{amscd}
\usepackage{hyperref}
\usepackage{enumerate}

\newcommand{\R}{\mathrm{R}}
\newcommand{\sF}{\mathcal{F}}
\newcommand{\F}{\mathbb{F}}
\renewcommand{\H}{\mathrm{H}}
\newcommand{\Z}{\mathbb{Z}}
\newcommand{\Q}{\mathbb{Q}}
\newcommand{\X}{\mathcal{X}}
\renewcommand{\P}{\mathbb{P}}
\newcommand{\mmu}{\boldsymbol{\mu}}
\newcommand{\Gm}{\mathbb{G}_\textrm{m}}
\renewcommand{\O}{\mathcal{O}}
\newcommand{\A}{\mathcal{A}}
\newcommand{\nr}{\mathrm{nr}}
\newcommand{\knr}{{k_\nr}}
\newcommand{\Xnr}{{\X_\nr}}
\newcommand{\Onr}{{\O_\nr}}
\newcommand{\Fb}{\bar{\F}}
\newcommand{\Xs}{\X_0}
\newcommand{\Xsns}{\Xs^\circ}
\newcommand{\Xns}{\X^\circ}
\newcommand{\Xsb}{\bar{\X}_0}
\newcommand{\Xsbns}{\Xsb^\circ}
\newcommand{\Zn}{\Z/n}
\newcommand{\isom}{\cong}

\newcommand{\Yt}{\tilde{Y}}
\newcommand{\nonp}{(p')}
\renewcommand{\t}{\bar{\partial}}
\newcommand{\p}{\mathfrak{p}}
\newcommand{\Picsh}{{\mathcal{P}\!\mathit{ic}}}
\newcommand{\ad}{\mathbf{A}}
\DeclareMathOperator{\Br}{Br}
\DeclareMathOperator{\Pic}{Pic}
\DeclareMathOperator{\coker}{coker}
\DeclareMathOperator{\Spec}{Spec}
\DeclareMathOperator{\Gal}{Gal}
\DeclareMathOperator{\cores}{cores}
\DeclareMathOperator{\Hom}{Hom}
\DeclareMathOperator{\NS}{NS}
\DeclareMathOperator{\res}{res}
\DeclareMathOperator{\inv}{inv}

\newtheorem{theorem}{Theorem}[section]
\newtheorem{lemma}[theorem]{Lemma}
\newtheorem{proposition}[theorem]{Proposition}
\newtheorem{corollary}[theorem]{Corollary}
\newtheorem{fact}[theorem]{Fact}
\theoremstyle{remark}
\newtheorem{remark}[theorem]{Remark}
\newtheorem{example}[theorem]{Example}
\theoremstyle{definition}
\newtheorem{definition}[theorem]{Definition}

\title{Bad reduction of the Brauer--Manin obstruction}
\author{Martin Bright}
\email{m.j.bright@math.leidenuniv.nl}
\address{Mathematisch Instituut, Niels Bohrweg 1, 2333 CA Leiden, Netherlands}


\begin{document}

\begin{abstract}
We relate the Brauer group of a smooth variety over a $p$-adic field to the geometry of the special fibre of a regular model, using the purity theorem in \'etale cohomology.  As an illustration, we describe how the Brauer group of a smooth del Pezzo surface is determined by the singularity type of its reduction.  We then relate the evaluation of an element of the Brauer group to the existence of points on certain torsors over the special fibre; we use this to describe situations when the evaluation is constant, and situations when the evaluation is surjective.  In the latter case, we describe how this surjectivity can be used to prove vanishing of the Brauer--Manin obstruction on varieties over number fields.
\end{abstract}

\maketitle

\section{Introduction}

\subsection{Background}

An important tool for studying rational points on surfaces and higher-dimensional varieties is the \emph{Brauer--Manin obstruction}.  This is an obstruction to the existence of rational points, first constructed by Manin~\cite{Manin:GBG}, based on the Brauer group of the variety in question.  For background information on Brauer groups of schemes, see the three articles by Grothendieck~\cite{Grothendieck:GB1,Grothendieck:GB2,Grothendieck:GB3} or Chapter~IV of Milne's book~\cite{Milne:EC}.  Let $X$ be a smooth, geometrically irreducible variety over a number field $K$.  For each place $v$ of $K$, there is a pairing
\begin{equation}\label{eq:locpr}
\Br X \times X(K_v) \to \Br K_v
\end{equation}
given by evaluation: $(\A, P) \mapsto \A(P)$.  Let $\ad_K$ denote the ring of ad\`eles of $K$.  Applying the local invariant map and summing over all places of $K$, we obtain a pairing
\[
\Br X \times X(\ad_K) \to \Q/\Z
\]
defined by $(\A, (P_v)) \mapsto \sum_v \inv_v \A(P_v)$.  Denote by $X(\ad_K)^{\Br}$ the right kernel of this pairing, that is,
\[
X(\ad_K)^{\Br} = \big\{ (P_v) \in X(\ad_K) \quad \big| \quad \sum_v \inv_v \A(P_v) = 0
\text{ for all } \A \in \Br X \big\}.
\]
Manin's observation was that the set of rational points $X(K)$ must be contained, under the diagonal inclusion in $X(\ad_K)$, in $X(\ad_K)^{\Br}$.  In particular, if $X(\ad_K)^{\Br}$ is empty, then $X(K)$ is also empty, and we say that there is a \emph{Brauer--Manin obstruction} to the existence of $K$-rational points on $X$.  In some situations, the Brauer group $\Br X$ and the set $X(\ad_K)^{\Br}$ are effectively computable, making the Brauer--Manin obstruction a valuable tool for the explicit study of rational points.

Central to understanding the Brauer--Manin obstruction is the purely local problem of understanding the pairing~\eqref{eq:locpr}.  One way to view the local structure of a variety at a finite place $v$ is to look at the geometry of a regular model over the ring of integers at $v$.  In this article, we describe how the geometry of a regular model is reflected both in the Brauer group of the variety $X_v = X \times_K K_v$ and in the evaluation pairing.  As an application, we study two cases of bad reduction of del Pezzo surfaces: that where the special fibre is a singular del Pezzo surface, and that where the special fibre is a cone.

There are several results in the literature which we will recover and generalise.  Colliot-Th\'el\`ene and Skorobogatov~\cite{CTS:TAMS-2013} studied the case of good reduction.  In particular, they gave various conditions under which the evaluation map is constant at a place of good reduction.  In a remark, they outlined an extension of some of their results to the case of bad reduction; Section~\ref{sec:prelim} of this article describes very similar results.  On the other hand, several authors have used surjectivity of the evaluation map associated to an algebra to show that it cannot give an obstruction.  For example, Colliot-Th\'el\`ene, Kanevsky and Sansuc~\cite{CTKS} used this approach for diagonal cubic surfaces, and the present author has applied a similar technique to diagonal quartic surfaces~\cite{Bright:vanishing}.  In Sections~\ref{sec:surj} and~\ref{sec:surj2}, we will show how these proofs can be viewed as special cases of a very general technique.

\subsection{Overview}

Our main tool is the purity theorem for the Brauer group, discussed by Grothendieck in Section~6 of~\cite{Grothendieck:GB3}, which describes what happens to the Brauer group of a regular scheme when a regular subscheme is removed.  Section~2 is devoted to recalling the absolute purity theorem in \'etale cohomology and deriving several versions of purity for the Brauer group.  In Section~\ref{sec:arithmetic}, we apply the purity results to relate the Brauer group of a variety to that of a regular model; as an application, in Section~\ref{sec:sdp} we calculate the Brauer group of a del Pezzo surface having the mildest sort of bad reduction, that is, where the reduction is a singular del Pezzo surface.  After an unramified base change, the Brauer group is entirely determined by the singularity type of the bad reduction.  We use this result to exhibit a del Pezzo surface of degree $1$ having an element of order $5$ in its Brauer group.

In Section~\ref{sec:eval}, we turn our attention from the Brauer group itself to the evaluation pairing.  Restricting to a particular element of $\Br X$ gives an evaluation map $X(K_v) \to \Br K_v$.  We begin by giving some criteria for the evaluation map to be constant.  Then, in the opposite direction, we investigate when the image of the evaluation map is as large as possible; using the Weil conjectures, we prove that this happens whenever the residue field is sufficiently large.  In particular, any difference between the evaluation map on points and that on zero-cycles can only happen when the residue field is small.  This allows us to recover a result of Colliot-Th\'el\`ene and Saito characterising those elements of the Brauer group which vanish on all zero-cycles of degree one. In Section~\ref{sec:cone}, we describe in more detail the example of a smooth variety reducing to a cone.  In Section~\ref{sec:global}, we return to global fields and describe how the surjectivity results of Section~\ref{sec:eval} may be used to prove that certain varieties have no Brauer--Manin obstruction to the existence of rational points.  As an application, we give criteria for vanishing of the Brauer--Manin obstruction when a variety reduces to a cone at a bad place.

\subsection{Notation}

If $A$ is an Abelian group, then $A[n]$ and $A/n$ denote the kernel
and cokernel, respectively, of the multiplication-by-$n$ map on $A$; if $\ell$ is prime, then $A(\ell)$ denotes the $\ell$-power torsion subgroup of $A$, and $A(\ell')$ denotes the prime-to-$\ell$ torsion subgroup.

If $X$ is a scheme, then the \emph{Brauer group} of $X$, denoted $\Br
X$, is defined to be the \'etale cohomology group $\H^2(X, \Gm)$.

We will often identify the group $\Zn$ with the $n$-torsion subgroup of $\Q/\Z$, and is doing so identify $\H^1(X,\Zn)$ with the $n$-torsion subgroup of $\H^1(X,\Q/\Z)$.

\section{Variations on the purity theorem}

In~\cite[Section~6]{Grothendieck:GB3}, Grothendieck describes the purity theorem for the Brauer group: that is, the relationship between $\Br X$ and $\Br (X \setminus Z)$ when $Z$ is a regular subscheme of a regular scheme $X$.  In this section, we derive some variants of that result.  We refer to Section~I.16 of Milne's lecture notes~\cite{Milne:LEC} for notation and background.  The absolute cohomological purity theorem, proved by
Gabber~\cite{Fujiwara:PC,Riou:Gysin}, asserts the following:

\begin{theorem}[Absolute cohomological purity]
Let $X$ be a regular scheme, and let $Z$ be a regular closed subscheme
of $X$, everywhere of codimension $c$.  Write $i$ for the inclusion $Z
\to X$.  Let $n$ be coprime to the residue characteristics of $X$; let
$\Lambda$ denote the constant sheaf $\Zn$ on $X$.  Then we have $\R^r
i^! \Lambda=0$ for $r \neq 2c$, and there is a canonical isomorphism
$\R^{2c} i^! \Lambda \isom \Lambda(-c)$ of sheaves on $Z$.
\end{theorem}

A standard argument then shows that, for any finitely generated
locally free sheaf $\sF$ of $\Lambda$-modules on $X$, we have $\R^r i^!
\sF=0$ for $r \neq 2c$, and an isomorphism $\R^{2c} i^! \sF \isom i^*
\sF (-c)$.
As described by Milne~\cite[p.~112]{Milne:LEC}, the Purity Theorem gives rise to isomorphisms $\H^r_Z(X,\sF) \isom \H^{r-2c}(Z, \sF(-c))$ for all $r \ge 2c$, and $\H^r_Z(X,\sF)=0$ for $r<2c$.  Using these isomorphisms, the long exact sequence for the pair $(X,Z)$ shows that there are isomorphisms $\H^r(X,\sF) \isom \H^r(X \setminus Z, \sF)$ for $r \le 2c-2$, and a long exact sequence of $\Zn$-modules (the \emph{Gysin sequence}):
\begin{multline}\label{eq:les}
0 \to \H^{2c-1}(X, \sF) \to \H^{2c-1}(X \setminus Z, \sF) \to
\H^0(Z, \sF(-c)) \to \\
\to \H^{2c}(X, \sF) \to \H^{2c}(X \setminus Z, \sF) \to \H^1(Z,
\sF(-c)) \to \dotsb.
\end{multline}
The Gysin sequence is functorial in the pair $(X,Z)$.

\begin{remark}
In what follows, we do not need the full force of the purity theorem: we will only make use of $\R^r i^! \mmu_n$ for $r \le 3$.  For $r \le 2$, the calculation of these sheaves is standard (see for example~\cite[Cycle, Proposition~2.1.4]{SGA4.5}); so it is only the vanishing of $\R^3 i^! \mmu_n$ which matters to us.
\end{remark}

It will be helpful to deduce several corollaries where $Z$ is not
necessarily regular.  (Note that Grothendieck's statements of Theorem~6.1, and
therefore of Corollary~6.2, in~\cite{Grothendieck:GB3} are missing some necessary hypotheses, as pointed out in~\cite{DF:JA-1984}.)

\begin{corollary}[Semi-purity in codimension $\ge 2$]\label{cor:codim2}
Suppose that $X$ is a regular excellent scheme, and $Z$ any
closed subscheme of codimension $c \ge 2$.  Then the natural map
$\H^2(X, \sF) \to \H^2(X \setminus Z, \sF)$ is an isomorphism;
and the natural map $\H^3(X, \sF) \to \H^3(X \setminus Z,
\sF)$ is injective, and an isomorphism if $c\ge 3$.
\end{corollary}
\begin{proof}
The statement is unaffected by replacing $Z$ with the reduced subscheme $Z_\textrm{red}$, so we may assume that $Z$ is reduced.
Now we use descending induction on the codimension $c$ of $Z$.  If $c = \dim X$, then $\dim Z=0$, so $Z$ is regular
and we can apply the Purity Theorem directly.  Otherwise, let $S$ denote the non-regular locus of $Z$, together with
any components of $Z$ of codimension strictly greater than $c$; since
$Z$ is reduced and excellent, $S$ is a closed
subscheme of codimension at least $1$ in $Z$, so of codimension at least $c+1$ in $X$.  By induction, we have
isomorphisms
\[
\H^2(X, \sF) \cong \H^2(X \setminus S, \sF) \text{ and }
\H^3(X, \sF) \cong \H^3(X \setminus S, \sF).
\]
Now $Z \setminus S$ is a regular closed subscheme of $X \setminus S$,
and so a further application of the Purity Theorem gives the desired
result.
\end{proof}

\begin{corollary}[Semi-purity in codimension $1$]\label{cor:codim1}
Let $X$ be a regular excellent scheme, and $Z$ any reduced, closed
subscheme everywhere of codimension $1$.  Suppose that the non-regular
locus $S$ of $Z$ is of codimension $c$ in $Z$.  Write $X^\circ$ for $X
\setminus S$, and $Z^\circ$ for $Z \setminus S$.  Then there is an
exact sequence
\begin{multline*}
0 \to \H^1(X, \sF) \to \H^1(X \setminus Z, \sF) \to 
\H^0(Z, \sF(-1)) \to \\
\to \H^2(X, \sF) \to \H^2(X \setminus Z, \sF) \to
\H^1(Z^\circ, \sF(-1)) \to \\
\to \H^3(X^\circ, \sF) \to \H^3(X \setminus Z, \sF).
\end{multline*}
If $c \ge 2$, then we may replace $\H^3(X^\circ, \sF)$ with
$\H^3(X, \sF)$.
\end{corollary}
\begin{proof}
This comes from applying the Purity Theorem to $Z^\circ \subset
X^\circ$ to give an exact sequence as in~\eqref{eq:les}.  Now $X^\circ
\setminus Z^\circ$ is the same as $X \setminus Z$; and Corollary~\ref{cor:codim2} gives isomorphisms $\H^p(X^\circ, \sF) \cong
\H^p(X, \sF)$ for $p=1,2$, and also for $p=3$ in the case $c \ge 2$.
\end{proof}

\begin{corollary}\label{cor:brpurity}
Under the conditions of Corollary~\ref{cor:codim1}, there is an exact
sequence
\[
0 \to \Br X[n] \to \Br(X \setminus Z)[n] \to \H^1(Z^\circ, \Zn) \to
\H^3(X^\circ, \mmu_n) \to \H^3(X \setminus Z, \mmu_n)
\]
for any $n$ coprime to the residue characteristics of $X$.  If $c \ge 2$, then we may replace $\H^3(X^\circ, \mmu_n)$ with $\H^3(X, \mmu_n)$.
\end{corollary}
\begin{proof}
Applying Corollary~\ref{cor:codim1} with $\sF = \mmu_n$ produces
the exact sequence
\[
\H^2(X, \mmu_n) \to \H^2(X \setminus Z, \mmu_n) \to \H^1(Z^\circ,
\Zn) \to \H^3(X^\circ, \mmu_n) \to \H^3(X \setminus Z, \mmu_n).
\]
Now the Kummer sequence gives a commutative diagram with exact rows
\[
\begin{CD}
0 @>>> \Pic X / n @>>> \H^2(X, \mmu_n) @>>> \Br X[n] @>>> 0 \\
@.     @V{\alpha}VV    @V{\beta}VV          @V{\gamma}VV \\
0 @>>> \Pic (X \setminus Z) / n @>>> \H^2(X \setminus Z, \mmu_n) @>>>
\Br (X \setminus Z)[n]
@>>> 0.
\end{CD}
\]
Since $X$ is regular, the restriction map $\Pic X \to \Pic (X
\setminus Z)$ is surjective, and so $\alpha$ is surjective; therefore
$\coker \beta \cong \coker \gamma$.  Moreover, it follows
from~\cite[Corollary~1.8]{Grothendieck:GB2} that $\Br X \to \Br (X
\setminus Z)$ is injective, and therefore $\gamma$ is injective.
Putting these facts together gives the claimed result.
\end{proof}

\begin{corollary}[Semi-purity for the Brauer group]
\label{cor:brpurity2}
Under the conditions of Corollary~\ref{cor:codim1}, there is an exact
sequence
\begin{multline*}
0 \to \Br X(\ell) \to \Br(X/Z)(\ell) \to \H^1(Z^\circ, \Q_\ell/\Z_\ell) \to \\
\to \H^3(X^\circ, \Gm)(\ell) \to \H^3(X \setminus Z, \Gm)(\ell)
\end{multline*}
for any prime $\ell$ not dividing the residue characteristics of $X$.  If $c \ge 2$, then we may replace $\H^3(X^\circ, \Gm)$ with $\H^3(X, \Gm)$.
\end{corollary}
\begin{proof}
We apply Corollary~\ref{cor:brpurity} with $n=\ell^r$ for increasing $r$, and then take the direct limit to obtain an exact sequence
\begin{multline*}
0 \to \Br X(\ell) \to \Br(X/Z)(\ell) \to \H^1(Z^\circ, \Q_\ell/\Z_\ell) \to \\
\to \H^3(X^\circ, \mmu_{\ell^\infty}) \to \H^3(X \setminus Z, \mmu_{\ell^\infty}).
\end{multline*}
On the other hand, the Kummer sequence gives rise to an exact sequence
\[
0 \to \Br X^\circ \otimes \Q_\ell / \Z_\ell \to
\H^3(X^\circ, \mmu_{\ell^\infty}) \to \H^3(X^\circ, \Gm)(\ell)
\to 0
\]
as in~\cite[Th\'eor\`eme~3.1]{Grothendieck:GB2}, and similarly for $X \setminus Z$.  Since $X^\circ$ is regular, $\Br X^\circ$ is torsion by~\cite[Proposition~1.4]{Grothendieck:GB2} and therefore $\Br X^\circ \otimes \Q_\ell/\Z_\ell$ is trivial, giving an isomorphism $\H^3(X^\circ, \mmu_{\ell^\infty}) \isom \H^3(X^\circ, \Gm)(\ell)$.  Similarly, we obtain a compatible isomorphism $\H^3(X \setminus Z, \mmu_{\ell^\infty}) \isom \H^3(X \setminus Z, \Gm)(\ell)$, and hence the claimed result.
\end{proof}

\section{The arithmetic setting}\label{sec:arithmetic}

Throughout this section, we fix the following notation.  $k$ is a finite extension of $\Q_p$, with ring of integers $\O$ and residue field $\F$.  $\X$ is a regular scheme, faithfully flat, separated and of finite type over $\O$, such that the generic fibre $X = \X \times_\O \Spec k$ is a smooth, geometrically irreducible variety over $k$.  The special fibre is denoted by $\Xs = \X \times_\O \Spec \F$.  The non-singular locus of $\Xs$ is $\Xsns$, and the complement in $\X$ of the singular locus of $\Xs$ is denoted $\Xns$.
Because $\Xsns$ is a non-singular variety over $\F$, its irreducible
components are also its connected components; call them $Z_1, \dotsc, Z_r$.  

We will sometimes want to pass to the maximal unramified extension of $k$.  Let $\knr$ denote the maximal unramified extension of $k$, and $\Onr$
its ring of integers; the residue field is $\Fb$, an algebraic closure of $\F$.  Denote by $\Xnr$
the base change of $\X$ to $\Onr$, and by $\Xsbns$ the base change of $\Xsns$ to $\Fb$.

Under the assumptions above, we may apply Corollary~\ref{cor:brpurity2} to the pair $\Xs \subset \X$ and take the product over all primes $\ell \neq p$ to obtain the following exact sequence:
\begin{multline}\label{eq:es}
0 \to \Br \X \nonp \to \Br X \nonp \xrightarrow{\prod_i \partial_i}
\prod_i \H^1(Z_i, \Q/\Z) \nonp \to \\
\to \H^3(\Xns, \Gm)\nonp \to \H^3(X, \Gm)\nonp.
\end{multline}
In the special case $\X=\Spec \O$, we have $\Br\O=0$ and we simply
recover the prime-to-$p$ part of the residue map $\Br k \isom
\H^1(\F,\Q/\Z)$ of class field theory.

\begin{remark}
Colliot-Th\'el\`ene and Skorobogatov~\cite{CTS:TAMS-2013} made use of
the first three terms of~\eqref{eq:es}, but instead derived it from a complex defined by Kato using K-theory.
\end{remark}

\begin{remark}
If $\X$ is smooth and proper over $\O$, with geometrically irreducible
fibres, then the exact sequence~\eqref{eq:es} gives a split
short exact sequence
\[
0 \to \Br \X \nonp \to \Br X \nonp \xrightarrow{\partial} \H^1(\Xs, \Q/\Z) \nonp \to 0.
\]
This may be seen as follows, as pointed out by Olivier Wittenberg.  Let $c$ be an element of $\H^1(\Xs, \Q/\Z)$ of order coprime to $p$, which we think of as an element of $\H^1(\Xs, \Zn)$ for suitable $n$.
Proper base change shows that the natural map $\H^1(\X,\Zn) \to
\H^1(\Xs,\Zn)$ is an isomorphism; composing its inverse with the restriction map $\H^1(\X,\Zn) \to \H^1(X, \Zn)$ we obtain a map $\phi \colon \H^1(\Xs,\Zn) \to \H^1(X, \Zn)$.  Now let $\pi$ be a uniformising element in $k$, and denote also by $\pi$ its class in $\H^1(k, \mmu_n) \isom k^\times/(k^\times)^n$.  Then the cup product $\phi(c) \cup \pi$ lies in $\H^2(X, \mmu_n)$, and can be pushed forward to $\Br X$.  We claim that this construction gives a section of $\partial$.  To see this, let $K$ denote the function field of $X$ and $\hat{K}$ its completion under the valuation corresponding to $\Xs$, and let $K_0$ be the function field of $\Xs$.  Then there is a natural inclusion $\H^1(\Xs,\Q/\Z) \subset \H^1(K_0, \Q/\Z)$, and the result follows from the corresponding fact for $\Br \hat{K} \to \H^1(K_0, \Q/\Z)$, which goes back to Witt~\cite{Witt:JRAM-1937} (see also \cite[Exercise~XII.3.3]{Serre:LF} for the statement without the assumption that the residue field be perfect).
\end{remark}

The sequence~\eqref{eq:es} involves the cohomology of $\Xns$,
which in general may not be easy to understand.  However, when $\X$ is
proper over $\O$ and not too singular, we can use the proper base change theorem to relate
the cohomology groups of $\Xns$ to those of $\Xs$.  In doing so, we see
a strong relationship between the cohomology of $\Xs$ and the Brauer
group of $X$.

\begin{remark}
For the following proposition, we assume that the special fibre $\Xs$
is non-singular in codimension $1$.  Since $\X$ is regular and $\Xs$
is an effective Cartier divisor in $\X$, so $\Xs$ is Cohen--Macaulay.
By Serre's criterion, we deduce that $\Xs$ is normal and therefore
geometrically integral.  In particular, the special fibre of our model
can have only one component.
\end{remark}

\begin{proposition}\label{prop:proper}
Suppose that $\X$ is proper over $\O$,
that $\Xs$ is non-singular in codimension $1$, and that the natural map $\Pic\X \to \Pic\Xs$ is surjective.  Let $n$ be coprime to $p$.  Then there are
natural isomorphisms $\Br\X[n] \to \Br \Xs[n]$ and $\H^3(\Xns, \Zn) \to
\H^3(\Xs, \Zn)$, and so the sequence of Corollary~\ref{cor:brpurity} becomes
\[
0 \to \Br \Xs[n] \to \Br X[n] \xrightarrow{\partial} \H^1(\Xsns, \Zn) \to \H^3(\Xs, \mmu_n) \to \H^3(X, \mmu_n).
\]
\end{proposition}

\begin{proof}
The proper base change theorem shows that the natural maps $\H^i(\X,
\mmu_n) \to \H^i(\Xs, \mmu_n)$ are isomorphisms.  Now the Kummer sequence
gives a commutative diagram with exact rows as follows.
\[
\begin{CD}
0 @>>> \Pic \X/n @>>> \H^2(\X, \mmu_n) @>>> \Br\X[n] @>>> 0 \\
@. @VVV @VV{\isom}V @VVV \\
0 @>>> \Pic \Xs/n @>>> \H^2(\Xs, \mmu_n) @>>> \Br\Xs[n] @>>> 0 \\
\end{CD}
\]
Because the natural map $\Pic \X \to \Pic \Xs$ is surjective, the left-hand vertical
map is also surjective.  By the Snake Lemma, the right-hand vertical
map is an isomorphism.

Now, since $\Xs$ is non-singular in codimension $1$,
Corollary~\ref{cor:codim2} shows that the restriction map $\H^3(\X,
\mmu_n) \to \H^3(\Xns,\mmu_n)$ is an isomorphism; combining this with
the isomorphism from proper base change gives the desired result.
\end{proof}

\begin{remark}
The hypothesis that $\Pic\X \to \Pic\Xs$ be surjective is satisfied whenever $\Xs$ satisfies $\H^2(\Xs,\mathcal{O}_{\Xs})=0$: see~\cite[Corollaire~1 to Proposition~3]{Grothendieck:GFGA}.
\end{remark}

\begin{corollary}\label{cor:proper}
Assume further that $\Br \Xs$ is torsion.  Then the following sequence is exact:
\begin{multline*}
0 \to \Br\Xs\nonp \to \Br X\nonp \to \H^1(\Xsns, \Q/\Z)\nonp \to \\
\to \H^3(\Xs, \Gm)\nonp \to \H^3(X, \Gm)\nonp.
\end{multline*}
\end{corollary}
\begin{proof}
As in the proof of Corollary~\ref{cor:brpurity2}, the assumption that $\Br \Xs$ is torsion gives isomorphisms $\H^3(\Xs, \mmu_{\ell^\infty}) \isom \H^3(\Xs, \Gm)(\ell)$ and $\H^3(X, \mmu_{\ell^\infty}) \isom \H^3(X, \Gm)(\ell)$ for each prime $\ell \neq p$.  Putting these together gives the stated sequence.
\end{proof}

\section{Application: del Pezzo surfaces}\label{sec:sdp}

In this section we apply Corollary~\ref{cor:proper} to the situation where $X$ is a smooth del Pezzo surface over $k$, and
the reduction $\Xs$ is a \emph{singular del Pezzo surface} over $\F$,
in the sense of Coray and Tsfasman~\cite{CT:PLMS-1988}.  (For example, a singular del Pezzo surface of degree $3$ is a cubic surface having only rational double points as singularities.)  The purpose of this section is to
demonstrate how the geometry of the reduction of a variety controls
the Brauer group of the variety.  For background information on Brauer groups of smooth del Pezzo surfaces, including a list of all the possible Brauer groups, see~\cite{Corn:PLMS-2007}.  In particular, recall that the Brauer group of any rational variety over an algebraically closed field is trivial.

Our main theorem is the following statement concerning $\Br X_\nr$.

\begin{theorem}\label{thm:sdp}
Let $\X$ be a regular scheme, proper and flat over $\O$, such that the generic fibre $X$ is a smooth del Pezzo surface over $k$ and the special fibre $\Xs$ is a singular del Pezzo surface over $\F$.  Then the sequence
\[
0 \to \Br \Xsb\nonp \to \Br X_\nr\nonp \to H^1(\Xsbns, \Q/\Z)\nonp \to 0
\]
is exact.
\end{theorem}

Before proving the theorem, we extract some corollaries showing how the Brauer group of $X_\nr$ is determined by the reduction type.  The Brauer groups of singular del Pezzo surfaces were calculated in~\cite{Bright:brsing}.  It turns out that, if $Y$ is a singular del Pezzo surface over an algebraically closed field, then there is a (non-canonical) isomorphism $\Br Y \isom \Pic(Y^\circ)_\textrm{tors}$.  In particular, in our situation, the two groups $\Br \Xsb$ and $\H^1(\Xsbns, \Q/\Z)$ appearing in Theorem~\ref{thm:sdp} are always isomorphic.

\begin{corollary}\label{cor:dp4}
Suppose that $p>2$.  Let $X$ be a smooth del Pezzo surface of degree $4$ over $k$ admitting a flat, regular, proper model over $\O$, the special fibre of which is a singular del Pezzo surface of degree $4$.  Then $\Br X_\nr$ is isomorphic to $(\Z/2)^2$ if the special fibre has singularity type $2A_1+A_3$ or $4A_1$, and $\Br X_\nr$ is trivial otherwise.
\end{corollary}

\begin{proof}
The calculations of~\cite{Bright:brsing} show that $\Br \Xsb$ and $\H^1(\Xsbns, \Q/\Z)$ are both trivial, and so $\Br X_\nr$ is trivial by Theorem~\ref{thm:sdp}, except in the two special cases stated.  In the two special cases, Theorem~\ref{thm:sdp} gives an exact sequence
\[
0 \to \Z/2 \to \Br X_\nr\nonp \to \Z/2 \to 0.
\]
Swinnerton-Dyer~\cite{SD:MPCPS-1993} has listed the possible Brauer groups of del Pezzo surfaces of degree 4, and the only one which fits into this sequence is $(\Z/2)^2$.
\end{proof}

\begin{corollary}\label{cor:dp3}
Suppose that $p>3$.  Let $X$ be a smooth del Pezzo surface of degree $3$ over $k$ admitting a flat, regular, proper model over $\O$, the special fibre of which is a singular del Pezzo surface of degree $3$.  Then:
\begin{itemize}
\item $\Br X_\nr$ is trivial unless the special fibre has singularity type $A_1+A_5$, $2A_1+A_3$, $4A_1$ or $3A_2$;
\item if the special fibre has singularity type $A_1+A_5$, $2A_1+A_3$ or $4A_1$, then $\Br X_\nr$ is isomorphic to $(\Z/2)^2$;
\item if the special fibre has singularity type $3A_2$, then $\Br X_\nr$ is isomorphic to $(\Z/3)^2$.
\end{itemize}
\end{corollary}

\begin{proof}
As for the previous corollary, the isomorphic groups $\Br \Xsb$ and $\H^1(\Xsbns, \Q/\Z)$ are given in~\cite{Bright:brsing}; the list of possibilities for $\Br X_\nr$ given by Swinnerton-Dyer~\cite{SD:MPCPS-1993} then shows that, in each case, there is only one which fits with Theorem~\ref{thm:sdp}.
\end{proof}


One application of Theorem~\ref{thm:sdp} and the above corollaries is that it gives an approach to writing down explicit del Pezzo surfaces with particular Brauer groups, at least over local fields.  In the past, this problem has been approached from the point of view of constructing surfaces with a particular Galois action on the Picard group by various geometric techniques: see, for example, \cite{EJ:CEJM-2010} and~\cite{EJ:IJNT-2011}.  As a more extended application of Theorem~\ref{thm:sdp}, we now exhibit a smooth del Pezzo surface of degree $1$ over $\Q_{11}$ having a non-constant element of order $5$ in its Brauer group.  Although it is well known that such surfaces can exist, the author is not aware of any explicit examples in the literature.  

\begin{example}
Let $\X$ be the subscheme of weighted projective space $\P^3(1,1,2,3)$, with variables $x,y,z,w$, over the ring $\Z_{11}$, defined by the equation
\begin{multline}\label{eq:dp1}
w^2 = z^3 + (2x^4 + 5x^3y - 2x^2y^2 - 5xy^3 + 2y^4)z \\
+ (-3x^6 + 4x^5y - 4x^4y^2 - 4x^2y^4 - 4xy^5 + 8y^6).
\end{multline}
A straightforward calculation verifies that the generic fibre $X$ is smooth over $\Q_{11}$ and that the special fibre $\Xs$ has two singular points at $(x,y,z,w) = (\pm 1, 1, 5, 0)$.  Considering the singular points as points of $\X$, they are both of multiplicity $1$, and so $\X$ is regular.  Explicitly resolving the singularities (which takes two blow-ups) shows that they are both of type $A_4$, with the exceptional curves all individually defined over $\F_{11}$.  The calculation of~\cite{Bright:brsing} then shows that $\Br \Xsb$ and the torsion subgroup of $\Pic\Xsbns$ both have order $5$.  
Because $X \times \bar{\Q}_{11}$ has trivial Brauer group, the Hochschild--Serre spectral sequence shows that $\Br X_\nr$ is isomorphic to the Galois cohomology group $\H^1(\bar{\Q}_{11}/(\Q_{11})_\nr, \Pic(X \times \bar{\Q}_{11}))$.  Since the Galois action factors through a subgroup of the Weyl group of $E_8$, the only primes which might divide the order of $\Br X_\nr$ are $2$, $3$, $5$ and $7$.  So the prime-to-$11$ part of $\Br X_\nr$ is the whole of it.  Theorem~\ref{thm:sdp} gives the exact sequence
\begin{equation}\label{eq:br5}
0 \to \Z/5 \to \Br  X_\nr \to \Z/5 \to 0.
\end{equation}
Now the list of cases given by Corn~\cite[Theorem~4.1]{Corn:PLMS-2007} shows that the only possibility is that the sequence~\eqref{eq:br5} splits (as a sequence of Abelian groups), giving $\Br X_\nr \isom (\Z/5)^2$.

Over $\Q_{11}$, we cannot directly apply the same method, since the equivalent of Lemma~\ref{lem:H3-sing} does not hold.
However, to see that at least one non-trivial element of $\Br X_\nr$ descends to $\Q_{11}$, it is enough to show that $\Br \Xs$ is non-trivial and then apply Corollary~\ref{cor:proper}.  Let $\Gamma$ denote the absolute Galois group of $\F_p$.  The Hochschild--Serre spectral sequence for the Galois covering $\Xsb \to \Xs$ gives an exact sequence
\[
0 \to \Br \Xs \to (\Br \Xsb)^\Gamma \xrightarrow{\beta} \H^2(\F_p, \Pic \Xsb).
\]
Let $n$ be an integer coprime to $11$.  According to~\cite[Proposition~3.3]{CTS:JRAM-2013}, the map $\beta$ restricted to the $n$-torsion in $(\Br \Xsb)^\Gamma$ is the boundary map associated to the 4-term exact sequence
\[
0 \to \Pic \Xsb \xrightarrow{n} \Pic \Xsb \to \H^2(\Xsb,\mmu_n) \to \Br \Xsb [n] \to 0.
\]
As a consequence, any element of $(\Br \Xsb[n])^\Gamma$ that lifts to a $\Gamma$-invariant element of $\H^2(\Xsb, \mmu_n)$ lies in the kernel of $\beta$.

Let $f \colon Y \to \Xsb$ be a minimal resolution.  The Leray spectral sequence gives an exact sequence
\[
\H^0(\Xsb, \R^1 f_* \mmu_n) \to \H^2(\Xsb, \mmu_n) \to \H^2(Y, \mmu_n) \to \H^0(\Xsb, \R^2 f_* \mmu_n).
\]
Since the fibre of $f$ above each singular point of $\Xsb$ is a tree of $\P^1$s, we have $\R^1 f_* \mmu_n=0$ and therefore $\H^2(\Xsb, \mmu_n)$ injects into $\H^2(Y, \mmu_n)$.  As $Y$ is a rational variety, the Brauer group of $Y$ is trivial and so the Kummer sequence gives an isomorphism $(\Pic Y)/n \to \H^2(Y,\mmu_n)$.  But $\Pic Y$ is generated by the canonical class and the classes of the eight exceptional curves, which are all defined over $\F_{11}$; thus the Galois action on $\Pic Y$ is trivial, and we deduce that the Galois action on $\H^2(\Xsb, \mmu_n)$ is also trivial.  The Kummer sequence further shows that the Galois action on $\Br \Xsb[n]$ is trivial.  Applying this with $n=5$ shows that $\beta$ is the zero map, and the natural map $\Br \Xs \to (\Br \Xsb)^\Gamma = \Br \Xsb$ is an isomorphism.  Therefore we have $\Br \Xs \isom \Z/5$ and Corollary~\ref{cor:proper} shows that there is a $5$-torsion element in $\Br X$.

\end{example}

\begin{remark}
I thank Ronald van Luijk for providing me with the equation of an elliptic surface over $\P^1_\Q$ having two singular fibres of type $I_5$, which arose in the proof of~\cite[Proposition~5.3]{SvL:density}.  The equation~\eqref{eq:dp1} is simply the reduction of that equation modulo $11$.  Viewed as the equation of an elliptic surface in Weierstrass form, equation~\eqref{eq:dp1} defines a surface over $\F_{11}$ having two singular fibres of type $I_5$.  Viewing the equation instead as defining a surface in $\P^3(1,1,2,3)$ is equivalent to blowing down the zero section of the elliptic surface and gives a singular del Pezzo surface of degree $1$ with two $A_4$ singularities.
\end{remark}

Let us now prove Theorem~\ref{thm:sdp}.  We will apply Corollary~\ref{cor:proper} to $\Xnr$; the result will be established if we can show that $\H^3(\Xsb, \Gm)(\ell)$ is trivial for all primes $\ell \neq p$.  That calculation is accomplished by the following lemma.

\begin{lemma}\label{lem:H3-sing}
Let $Y$ be a singular del Pezzo surface over an algebraically closed
field.  Then, for any $n$ prime to the characteristic, we have $\H^3(Y, \Gm)[n]=0$.
\end{lemma}
\begin{proof}
Consider the resolution of singularities $f\colon \Yt \to Y$, where
$\Yt$ is the corresponding generalised del Pezzo surface (see~\cite{CT:PLMS-1988}).  Because the
sheaves $\R^q f_* \Gm$, for $q>0$, are supported at the finitely many
singular points of $Y$, so $\H^p(Y, \R^q f_* \Gm)$ vanishes for
$p,q>0$.  Also, $f$ is proper and birational, giving $f_* \Gm = \Gm$.  Using these facts, we extract from the Leray spectral sequence for $f$ an exact sequence
\[
\Br \Yt \to \H^0(Y, \R^2 f_* \Gm) \to \H^3(Y, \Gm) \to \H^3(\Yt, \Gm).
\]
Because $\Yt$ is a smooth, rational surface, we have $\Br \Yt = 0$.  Taking $n$-torsion in the above sequence, we have
\[
0 \to \H^0(Y, \R^2 f_* \Gm)[n] \to \H^3(Y, \Gm)[n] \to \H^3(\Yt, \Gm)[n].
\]
Now, $\Pic \Yt$ is torsion-free, giving $\H^1(\Yt, \mmu_n)=0$; by Poincar\'e duality, we deduce $\H^3(\Yt, \mmu_n)=0$ as well.  The Kummer sequence then shows $\H^3(\Yt, \Gm)[n]=0$.  To conclude, it will be enough to prove that $\R^2 f_* \Gm[n] = 0$, and we can check this on stalks as follows.

The Kummer sequence on $\Yt$ gives a short exact sequence of sheaves on $Y$:
\[
0 \to (\R^1 f_* \Gm)/n \to \R^2 f_* \mmu_n \to (\R^2 f_* \Gm)[n] \to 0.
\]
Let $y$ be a singular point of $Y$.  Taking stalks at $y$, which is exact and so preserves kernels and
cokernels, gives the top exact row of the following commutative diagram; the bottom row comes from the Kummer sequence on the fibre $\Yt_y$, which is a union of exceptional curves.
\[
\begin{CD}
0 @>>> (\R^1 f_* \Gm)_y/n @>>> (\R^2 f_* \mmu_n)_y @>>> (\R^2 f_*
\Gm)_y[n] @>>> 0 \\
@. @VVV @VVV @VVV \\
0 @>>> \Pic \Yt_y / n @>>> \H^2(\Yt_y, \mmu_n) @>>> \Br \Yt_y[n] @>>> 0
\end{CD}
\]
The middle vertical map is an isomorphism, by the proper base change
theorem.  Now $(\R^1 f_* \Gm)_y$ is the Picard group of $\Yt \times_Y \O_{Y,y}^\textrm{sh}$, the base change of $\Yt$ to the strictly Henselian local ring at $y$; so the left-hand vertical map in the diagram is surjective by~\cite[Lemma~14.3]{Lipman:IHES-1969}.  By the Snake Lemma, the right-hand vertical arrow is an isomorphism.  But, because $\Yt_y$ is a curve, we have $\Br \Yt_y=0$ by~\cite[Corollary~1.2]{Grothendieck:GB3}; note that this corollary does indeed apply to reducible curves.  It follows that $(\R^2 f_* \Gm)[n]$ vanishes, and we deduce $H^3(Y,\Gm)[n]=0$ as claimed.
\end{proof}

\section{Evaluating elements of Brauer groups}\label{sec:eval}

We return to the general situation of Section~\ref{sec:arithmetic}.
Given an element $\A \in \Br X$, we obtain an \emph{evaluation map}
$\A \colon X(k) \to \Br k$ by thinking of a $k$-point $P$ as a
morphism $\Spec k \to X$ and defining $\A(P) = P^* \A$.  In this
section, we show how the boundary maps $\partial$ of the Gysin
sequence allow us to describe the evaluation map in terms of the
special fibre of the given model.

\subsection{Preliminaries}\label{sec:prelim}

The natural map $\X(\O) \to X(k)$ is not necessarily surjective; in
other words, there may be $k$-points of $X$ which do not extend to
$\O$-points of our model.  We will not be able to say anything about
such points.  As $\X$ is regular, no point of $\X(\O)$ reduces to a
singular point of $\Xs$; that is, we have $\X(\O) = \Xns(\O)$.  

Denote by $\X_i$ the open subscheme of $\Xns$ obtained by removing all
components of the special fibre apart from $Z_i$.  Then $\X(\O) =
\bigcup_i \X_i(\O)$.
The results in this section describe in more detail the
relationship between $\partial_i\A$ and the evaluation map $\A \colon
\X_i(\O) \to \Br k$.  Since this may be studied separately on each
component of the special fibre, we may replace $\X$ by some $\X_i$ and
therefore assume from now on that \textbf{the non-singular locus of
  the special fibre $\Xsns$ has only one connected component}.  In this case, the exact sequence~\eqref{eq:es} gives
\begin{equation}\label{eq:es1}
0 \to \Br \X \nonp \to \Br X \nonp \xrightarrow{\partial}
\H^1(\Xsns, \Q/\Z) \nonp.
\end{equation}
We will further assume that \textbf{$\X(\O)$ is non-empty}, since most of
our statements will be vacuous otherwise.  Note that this condition implies that $\Xsns$ is geometrically connected.   We will sometimes identify $\X(\O)$ with its image in $X(k)$.

\begin{proposition}\label{prop:eval}
Let $P$ be a $k$-point of $X$ lying in the image of $\X(\O)$ and reducing to a point $P_0$ of $\Xsns$.  Then the following diagram commutes:
\[
\begin{CD}
\Br X \nonp @>{\partial}>> \H^1(\Xsns, \Q/\Z) \nonp \\
@V{P^*}VV               @V{P_0^*}VV \\
\Br k \nonp @>{\isom}>> \H^1(\F, \Q/\Z) \nonp
\end{CD}
\]
where the bottom map is the residue map of class field theory.
\end{proposition}

\begin{proof}
By assumption, $P$ extends to a morphism $\Spec \O \to \X$.  The diagram now
follows from the functoriality of the Gysin sequence applied to both $\X$ and $\O$.
\end{proof}

The following corollaries are immediate.

\begin{corollary}\label{cor:modp}
Let $\A$ be an element of $\Br X \nonp$.  Then $\A(P)$ depends only on $P_0$; that is, the evaluation map $\A \colon \X(\O) \to \Br k$ factors through  reduction mod $p$. 
\end{corollary}

\begin{corollary}\label{cor:residue}
For any $\A \in \Br X \nonp$, the evaluation map $\A \colon \X(\O) \to \Br k$ is determined by the class of $\partial \A$ in $\H^1(\Xsns, \Q/\Z)$. 
\end{corollary}

Corollaries~\ref{cor:modp} and~\ref{cor:residue} may be summarised by saying that the following diagram commutes, where the top row is the evaluation pairing for the Brauer group and the bottom row is the pairing defined by taking the isomorphism class of the fibre of a torsor at a point.

\[
\begin{array}{ccccl}
\Br X \nonp & \times & \X(\O) & \longrightarrow & \Br k \nonp \\[1.5ex]
\partial \Big\downarrow \quad & & \Big\downarrow & & \quad \Big\downarrow \\[1.5ex]
\H^1(\Xsns, \Q/\Z)\nonp & \times & \Xsns(\F) & \longrightarrow & \H^1(\F,\Q/\Z) \nonp
\end{array}
\]

\begin{corollary}
If $\X$ is any model of a smooth $k$-variety $X$, $\A$ an element of
$\Br X \nonp$, and $P$ a $k$-point of $X$ reducing to a smooth point $P_0$ of the special fibre, then $\A(P)$ depends only on the reduction $P_0$.
\end{corollary}
\begin{proof}
Apply Proposition~\ref{prop:eval} to a regular Zariski neighbourhood of $P_0$ in $\X$.
\end{proof}

In order to say more about how the evaluation map associated to an algebra $\A$ depends on the class of $\partial \A$, it will be useful to pass to the maximal unramified extension of $k$.  We keep the notation of Section~\ref{sec:arithmetic}.  Because $\Xnr$ is \'etale over $\X$
and therefore regular, we can consider the exact
sequence~\eqref{eq:es1} with $\X$ replaced by $\Xnr$.  Since we are assuming that $\Xsns(\F)$ is non-empty, it follows that $\Xsbns$ is connected.

\begin{lemma}\label{lem:nr}
Let $\Br(\Xnr/\X)$ denote the kernel of the natural map $\Br\X \to
\Br\Xnr$, and similarly let $\Br(X_\nr/X)$ denote the kernel of the
natural map $\Br X \to \Br X_\nr$.  Then
there is a commutative diagram, with exact rows and columns, as follows.
\begin{equation}\label{eq:nrdiag}
\begin{CD}
0 @>>> \Br(\Xnr/\X) \nonp @>>> \Br(X_\nr/X)\nonp @>{\partial}>> \H^1(\F, \Q/\Z)\nonp \\
@. @VVV @VVV @VVV \\ 
0 @>>> \Br \X \nonp @>>> \Br X \nonp @>{\partial}>> \H^1(\Xsns, \Q/\Z)\nonp \\
@. @VVV @VVV @VVV \\
0 @>>> \Br \Xnr \nonp @>>> \Br X_\nr \nonp @>>> \H^1(\Xsbns, \Q/\Z)\nonp
\end{CD}
\end{equation}
The top row is split, giving a canonical isomorphism 
\[
\Br(X_\nr/X)\nonp \cong \Br(\Xnr/\X)\nonp \oplus \Br k \nonp.
\]
\end{lemma}
\begin{proof}
The second and third rows simply come from the functoriality of our
exact sequence~\eqref{eq:es1}.  Now the Hochschild--Serre spectral
sequence applied to $\Xsbns \to \Xsns$ gives an exact sequence
\[
0 \to \H^1(\F, \Q/\Z) \to \H^1(\Xsns, \Q/\Z) \to \H^1(\Xsbns, \Q/\Z),
\]
identifying $\H^1(\F,\Q/\Z)$ with the subgroup of $\H^1(\Xsns, \Q/\Z)$
consisting of the isomorphism classes of constant torsors.  Therefore
the top row of our diagram consists of the kernels of the three
vertical base change maps, and so the whole diagram commutes.

Now $\Br k$ is certainly contained in $\Br(X_\nr/X)$, and class field
theory shows that the restricted map $\partial \colon \Br k \nonp \to
\H^1(\F,\Q/\Z) \nonp$ is an isomorphism; its inverse gives a canonical section
$\H^1(\F,\Q/\Z) \nonp \to \Br(X_\nr/X) \nonp$ whose image is $\Br k \nonp$.  This
shows that the top row is a split exact sequence, and identifies
$\Br(X_\nr/X) \nonp$ as a direct sum as claimed.
\end{proof}

Denote by $\t \colon \Br X \nonp \to \H^1(\Xsbns, \Q/\Z) \nonp$ the composite homomorphism from the diagram~\eqref{eq:nrdiag}.

\begin{proposition}\label{prop:const}
Let $\A$ be an element of $\Br X \nonp$, and suppose that $\t(\A)$ is zero.  Then the evaluation map $\A \colon \X(\O) \to \Br k$ is constant.
\end{proposition}
\begin{proof}
Lemma~\ref{lem:nr} shows that assuming $\t(\A)=0$ implies that $\partial(\A)$ is
the class of a constant torsor.  By Proposition~\ref{prop:eval}, the
evaluation map is constant on $\X(\O)$.
\end{proof}

\begin{corollary}
Let $n$ be coprime to $p$.  If $\H^1(\Xsbns, \Zn)$ is trivial then, for every
$\A \in \Br X[n]$, the evaluation map $\A \colon \X(\O) \to \Br k$ is
constant. 
\end{corollary}

\begin{remark}
Applying this to the del Pezzo surfaces of Section~\ref{sec:sdp}, we obtain the following:
\begin{itemize}
\item Let $X$ be a del Pezzo surface of degree $4$ satisfying the conditions of Corollary~\ref{cor:dp4}.  If the singularity type of the special fibre $\Xs$ is neither $2A_1+A_3$ nor $4A_1$ then, for any $\A \in \Br X$, the evaluation map associated to $\A$ is constant.
\item Let $X$ be a del Pezzo surface of degree $3$ satisfying the conditions of Corollary~\ref{cor:dp3}.  If the singularity type of the special fibre $\Xs$ is neither $A_1+A_5$, $2A_1+A_3$, $4A_1$ nor $3A_2$ then, for any $\A \in \Br X$, the evaluation map associated to $\A$ is constant.
\end{itemize}
\end{remark}

Proposition~\ref{prop:const} is unsurprising if $\A$ is either in $\Br
k$, or in $\Br\X$ (in the latter case, the evaluation map factors through $\Br \O$, which is trivial).  The following proposition states that those
classes generate $\ker\t$.

\begin{proposition}\label{prop:kert}
The kernel of $\t$ is $\Br \X\nonp \oplus \Br k\nonp \subset \Br X\nonp$.
\end{proposition}
\begin{proof}
$\t$ is defined to be composition of two homomorphisms, so we can
  describe its kernel using the kernel-cokernel exact sequence.
  Combined with the diagram of Lemma~\ref{lem:nr}, this gives an exact
  sequence
\[
0 \to \Br\X\nonp \to \ker \t \xrightarrow{\partial} \H^1(\F,\Q/\Z)\nonp.
\]
Since $\partial$ maps $\Br k\nonp \subseteq \ker\t$ isomorphically to
$\H^1(\F,\Q/\Z)\nonp$, this sequence splits.
\end{proof}

\subsection{Surjectivity of evaluation maps}\label{sec:surj}

It is natural to ask to what extent the converse of Proposition~\ref{prop:const} is true: if the evaluation of an algebra $\A$ is constant on $\X(\O)$, under what circumstances can we deduce that $\t(\A)=0$?  Certainly, the converse does not always hold: for example, if $\Xsns(\F)$ has only one point, then the evaluation map will be constant no matter what $\t(\A)$ is.  However, we shall see that this phenomenon can only happen for small residue fields $\F$, in the following sense: whenever $\F$ is sufficiently large, the evaluation map associated to an Azumaya algebra takes all possible values on points of $\X(\O)$.  Here ``takes all possible values'' means that, up to the class of a constant algebra, the image of the evaluation map is all of $\Br k[m]$, where $m$ is the order of $\t(\A)$; and the meaning of ``sufficiently large'' depends on $m$ and geometric invariants of $\Xsns$.

In order to simplify the statements of the remaining results in this section, we would like to ignore constant algebras.  A convenient way to do that is given by the following definition.

\begin{definition}
An element $\A \in \Br X$ is \emph{normalised} if there exists a point $P \in \X(\O)$ such that $\A(P) = 0$.
\end{definition}

Note that this condition is not vacuous, since by assumption $\X(\O)$
is non-empty.  Any choice of $P \in \X(\O)$ gives a retraction of the
inclusion $\Br k \to \Br X$; letting $\Br_P X$ denote the subgroup of
those $\A \in \Br X$ satisfying $\A(P)=0$, we obtain an isomorphism
$\Br X \isom \Br k \oplus \Br_P X$.  Thus any element of $\Br X$ may
be expressed as an element of $\Br k$ plus a normalised element of
$\Br X$.

\begin{lemma}\label{lem:norm}
Let $\A$ be a normalised element of $\Br X \nonp$, and suppose that $\t\A$ has order $n$ in $\H^1(\Xsbns, \Q/\Z)$.  Then $\partial\A$ has order $n$ in $\H^1(\Xsns, \Q/\Z)$, and the evaluation map $\A \colon \X(\O) \to \Br k$ takes values in $\Br k[n]$.
\end{lemma}

\begin{proof}
By hypothesis, $\t(n\A)$ is the trivial class in $\H^1(\Xsbns, \Q/\Z)$; as in the proof of Proposition~\ref{prop:const}, we deduce that $\partial(n\A)$ lies in the image of $\H^1(\F, \Q/\Z)$.  But $n\A$ is normalised, and so Proposition~\ref{prop:eval} shows that $\partial(n\A)$ is trivial.  Therefore $\partial\A$ is of order dividing $n$ in $\H^1(\Xsns, \Q/\Z)$.  Since $\t\A$ is of order exactly $n$, we deduce that $\partial\A$ is also of order exactly $n$.  The final statement follows immediately from Proposition~\ref{prop:eval}.
\end{proof}

Let us now notice that whether the evaluation map takes any particular
value is controlled by the existence of points on a certain variety
over $\F$.  Recall that any class in $\H^1(\Xsns, \Zn)$ is represented
by a torsor $Y \to \Xsns$ under $\Zn$.  Given an element $\alpha \in
\Br k[n]$, we obtain a class $\partial \alpha \in \H^1(\F, \Zn)$;
twisting $Y \to \Xsns$ by Galois descent gives a new torsor
$Y^{\partial \alpha} \to \Xsns$, geometrically isomorphic to $Y$.  The
class of $Y^{\partial\alpha}$ in $\H^1(\Xsns, \Zn)$ is simply the sum
of the classes of $Y$ and $\partial\alpha$.  (See~\cite{Skorobogatov:TRP}, Section~2.2 and particularly p.~22.)

\begin{lemma}\label{lem:torsors}
Let $\A$ be an element of $\Br X \nonp$, and suppose that
$\partial\A \in \H^1(\Xsns,\Q/\Z)$ has order $n$.  Let $Y \to \Xsns$
be a torsor representing the class of $\partial \A$ in $\H^1(\Xsns,
\Zn)$.  Let $\alpha$ be a class in $\Br k[n]$, and let $P \in \X(\O)$
be a point reducing to $P_0 \in \Xsns(\F)$.  Then $\A(P)$ is equal to $-\alpha$ if
and only if $P_0$ lies in the image of the twisted torsor
$Y^{\partial\alpha}(\F) \to \Xsns(\F)$.  In particular, $-\alpha$ is
in the image of the evaluation map $\A\colon \X(\O) \to \Br k$ if and
only if $Y^{\partial\alpha}(\F)$ is non-empty.
\end{lemma}

\begin{proof}
By Proposition~\ref{prop:eval}, we have $\A(P)=-\alpha$ if and only if the class in $\H^1(\F,\Zn)$ of the fibre $Y_{P_0}$ is $\partial\alpha$.  Twisting by $\partial\alpha$, this is true if and only if the class of the fibre $(Y^{\partial\alpha})_{P_0}$ is trivial.  This happens precisely when that fibre contains an $\F$-rational point, that is, when $P_0$ lies in the image of $Y^{\partial\alpha}(\F) \to \Xsns(\F)$.
\end{proof}

Lemma~\ref{lem:torsors} shows that, to prove that the evaluation map associated to an Azumaya algebra takes a particular value, it is enough to prove existence of rational points on a certain torsor over $\Xsns$.  A method to guarantee the existence of points on varieties over finite fields comes from the Weil conjectures.
\begin{fact}\label{weil}
Let $Y$ be a geometrically irreducible variety over a finite field
$\F$, and denote by $\bar{Y}$ the base change of $Y$ to the algebraic
closure of $\F$.  Then there is a bound $B(\bar{Y})$, depending only on
cohomological invariants of $\bar{Y}$, such that, whenever $|\F| >
B(\bar{Y})$, then $Y(\F)$ is non-empty.
\end{fact}

This fact follows from the generalised form of the Weil conjectures proved by Deligne~\cite{Deligne:WeilII} together with the Lefschetz trace formula.  More specifically, the bound depends on the Betti numbers of $\bar{Y}$ in compactly supported $\ell$-adic cohomology.

\begin{remark}
The usefulness of this fact is that the bound $B(\bar{Y})$ depends not
on the precise geometry of $\bar{Y}$, but only on certain geometric invariants.
The existence of such a bound is the famous result of Lang and Weil~\cite{LW:AJM-1954}, who use much more elementary methods than those of Deligne.  They show that, for a projective variety in $\P^n$, a bound can be chosen that depends only on $n$ and the dimension and degree of the variety.  For many of our applications, and in particular for all the corollaries to Theorem~\ref{thm:weil} below, this is enough.
The only reason for appealing to Deligne instead of Lang--Weil is that many natural applications of the results in this section arise in the context of a family of varieties, and Deligne's statements often immediately show how the bound $B(Y)$ varies as $Y$ moves in a family.  For example, Betti numbers are constant in a smooth, proper family of varieties, and so we see immediately that the same bound will work for every member of such a family.  
\end{remark}




In order to apply Fact~\ref{weil} to torsors, we need to understand
when a torsor is geometrically irreducible.  Let $S$ be a scheme, and $G$ a finite
Abelian group.  Recall that the
cohomology group $\H^1(S, G)$ classifies $S$-torsors under $G$;
these are the same as Galois covers of $S$ with group $G$, and so
are also classified by homomorphisms from $\pi_1(S)$ to $G$.  This
correspondence gives an isomorphism of groups $\H^1(S,G) \cong
\Hom(\pi_1(S),G)$.
\begin{lemma}\label{lem:conn}
Suppose that $S$ is normal and connected.  Then, under the isomorphism $\H^1(S,G) \cong \Hom(\pi_1(S),G)$, the
\emph{connected} torsors correspond to the \emph{surjective}
homomorphisms.  In particular, if $G$ is the cyclic group $\Zn$, then
the connected torsors are those representing classes of exact order
$n$ in $\H^1(S,G)$.
\end{lemma}

\begin{proof}
Let $K$ denote the function field of $S$.  Because $S$ is normal,
taking generic fibres induces a one-to-one correspondence from finite
\'etale covers of $S$ to \'etale algebras over $K$, preserving connectedness
(see~\cite[Expos\'e I, Section~10]{SGA1}); so it suffices to prove the
theorem when $S$ is the spectrum of a field.  In this case, let $T \to
K$ be a torsor under $G$, and recall the definition of the associated
element $\phi \in \Hom(\Gal(\bar{K}/K),G)$: pick a geometric point $x$
of $T$; then, for $\sigma \in \Gal(\bar{K}/K)$, we define
$\phi(\sigma)$ to be the unique $g \in G$ such that $\sigma x = gx$.
From this description it is easy to see that $\phi$ is surjective if
and only if the Galois action on the geometric points of $T$ is
transitive, that is, if and only if $T$ is connected.
\end{proof}

Armed with this, we can prove a partial converse to
Proposition~\ref{prop:const}.  The following theorem is a special case
of Theorem~\ref{thm:indep} below, but we state and prove it separately both
because it motivates the statement and proof of
Theorem~\ref{thm:indep} and because it has corollaries of independent
interest.

\begin{theorem}\label{thm:weil}
Fix a class $c \in \H^1(\Xsbns,\Q/\Z)$, of order $n$ coprime to
$p$.  Then there is a constant $B$, depending only on $c$, such
that, whenever $|\F| > B$ and for all normalised $\A \in \Br X$ such that $\t\A =
c$, the evaluation map $\A \colon \X(\O) \to \Br k[n]$ is
surjective.  If $\bar{Y} \to \Xsbns$ is a torsor representing the
class of $c$ in $\H^1(\Xsbns,\Zn)$, then we can take $B$ to be $B(\bar{Y})$.
\end{theorem}

\begin{proof}
By Lemma~\ref{lem:norm}, $\partial\A$ has order $n$ in $\H^1(\Xsns,
\Q/\Z)$, and so we may think of $\partial\A$ as lying in $\H^1(\Xsns,
\Zn)$.  Let $Y \to \Xsns$ be a torsor representing the class of
$\partial \A$ in $\H^1(\Xsns, \Zn)$.  Then $\bar{Y} \to \Xsbns$ is a
torsor representing $c$ in $\H^1(\Xsbns, \Zn)$; by hypothesis, $c$ is
of order $n$, and so Lemma~\ref{lem:conn} shows that $\bar{Y}$ is connected.  Since $\bar{Y}$ is also
smooth, $\bar{Y}$ is irreducible.  If $|\F|>B(\bar{Y})$,
then Fact~\ref{weil} shows that $Y$ has an $\F$-point, as does every
Galois twist of $Y$.  By Lemma~\ref{lem:torsors} the evaluation map
$\A\colon \X_i(\O) \to \Br k[n]$ is surjective.
\end{proof}

\begin{corollary}\label{cor:weil2}
Fix an integer $n$ coprime to $p$.  Then there is a constant
$B$, depending only on $\Xsbns$ and $n$, such that, whenever $|\F| >
B$ and for all normalised $\A \in \Br X$ such that $\t\A$ has order $n$,
the evaluation map $\A \colon \X(\O) \to \Br k[n]$ is surjective.
\end{corollary}

\begin{proof}
The group $\H^1(\Xsbns, \Zn)$ is finite~\cite[VI,
  Corollary~5.5]{Milne:EC}.  So we may apply Theorem~\ref{thm:weil} to
each class $c \in \H^1(\Xsbns, \Zn)$ and take $B$ to be as large as
necessary for the conclusion to hold in all cases.
\end{proof}

\begin{corollary}\label{cor:ext}
Let $\A$ be a normalised element of $\Br X \nonp$,
and suppose that $\t\A$ is of order $n$ in $\H^1(\Xsbns, \Q/\Z)$.
If $\ell$ is an unramified extension of $k$ of sufficiently large
degree, with ring of integers $\O_\ell$, then the evaluation map
$\A\colon \X(\O_\ell) \to \Br k[n]$ is surjective.
\end{corollary}
\begin{proof}
Replacing $\X$ by its base change to $\O_\ell$ changes neither
$\Xsbns$ nor $\t\A$.  So Theorem~\ref{thm:weil} shows that the
conclusion will hold as long as the size of the residue field of
$\ell$ is greater than $B$.
\end{proof}

We can use Corollary~\ref{cor:ext} to obtain further corollaries about zero-cycles on $X$; 
first let us briefly recall some notation.  A \emph{zero-cycle} on a variety $X$ is a finite formal linear combination $\sum n_i P_i$ of closed points $P_i$.  The group of zero-cycles on $X$ is denoted $Z_0(X)$.  Given $\A \in \Br X$, we define $\A(\sum n_i P_i) = \sum n_i \cores_{k(P_i)/k} \A(P_i)$, giving a pairing $\Br X \times Z_0(X) \to \Br k$.  
The \emph{degree} of the zero-cycle $\sum n_i P_i$ is $\sum n_i[k(P_i):k]$.  Let $Z_0^0(X)$ denote the group of zero-cycles of degree $0$.  If $\A \in \Br k$, then we have $\A(z)=0$ for any $z \in Z_0^0(X)$; so we obtain a pairing $(\Br X / \Br k) \times Z_0^0(X) \to \Br k$.  
For each of these pairings, we may try to describe the left kernel.

In our situation, we need to talk about the group of zero-cycles on $X$ which extend to a section of $\X$, and so specialise to a point of the special fibre.  Denote this group by $Z_0(X,\X)$.  Equivalently, $Z_0(X,\X)$ is the image of the (injective) natural map $Z_1(\X) \to Z_0(X)$.

\begin{corollary}\label{cor:z0}
Let $\A$ be an element of $\Br X\nonp$, and suppose that $\t(\A)$ has order $n$ in $\H^1(\Xsbns, \Q/\Z)$.  Then the image of the evaluation map $\A \colon Z_0^0(X,\X) \to \Br k$ is $\Br k[n]$.
\end{corollary}
\begin{proof}
By assumption, $\X(\O)$ is non-empty; let $P \in X(k)$ be a $k$-point of $X$ that extends to an $\O$-point of $\X$.  Replacing $\A$ by $\A - \A(P)$, which affects neither the hypotheses nor the conclusion of the corollary, we reduce to the case in which $\A(P)=0$, that is, $\A$ is normalised at $P$.  Let $k_d$ be the unramified extension of $k$ of degree $d$.  Corollary~\ref{cor:ext} shows that, assuming $d$ to be sufficiently large, the evaluation map $\A \colon \X(\O_{k_d}) \to \Br k_d[n]$ is surjective.  It is straightforward to check that evaluating $\A$ at a point of $X(k_d)$ gives the same element of $\Br k_d$ as evaluating $\A$ at the corresponding closed point of $X$.  That the corestriction map from $\Br k_d$ to $\Br k$ is an isomorphism follows easily from~\cite[Chapter~XIII, Proposition~7]{Serre:LF}.  We deduce that the evaluation map associated to $\A$ takes the subset of zero-cycles of degree $d$ in $Z_0(X,\Z)$ surjectively onto $\Br k[n]$.  Subtracting the zero-cycle $dP$ gives the same result for zero-cycles of degree zero, as claimed.
\end{proof}

\begin{remark}
As we have seen, an element of $\Br X$ can be evaluated both at points of $X$ and at zero-cycles.  The combined message of Theorem~\ref{thm:weil} and Corollary~\ref{cor:z0} is that the images of these two evaluation maps can only differ when the residue field is small.  Applying this to varieties over global fields, we see that any difference between the Brauer--Manin obstruction to rational points and that to zero-cycles of degree one happens at small primes.  As an illustration, consider the diagonal quartic surface $X \subset \P^3_\Q$, defined by the equation
\[
X_0^4 + 47 X_1^4 = 103 X_3^4 + (17 \times 47 \times 103) X_4^4,
\]
described in Proposition~3.3 of~\cite{Bright:vanishing}.  There, it is shown that $X$ has points everywhere locally, that $\Br X / \Br \Q$ has order $2$, and that the non-trivial class gives an obstruction to the existence of rational points on $X$.  More explicitly, the non-constant class in $\Br X$ is represented by the algebra $\A = (17,f/X_0^2)$, where $f$ is the polynomial $(20 X_0^2 + (47\times 13) X_1^2 + (103 \times 9) X_2^2)$.  It is easy to check that the invariant map of $\A$ is zero at all places apart from $17$.  At $17$, the special fibre $\Xs$ of the given model is the cone over a plane quartic curve; we find that $\partial \A \in \H^1(\Xsns,\Z/2\Z)$ is non-constant.  Let $\tilde{f}$ denote the reduction of $f$ modulo $17$; the variety defined over $(\Xs \setminus \{ \tilde{f}=0 \})$ by $T^2 = \tilde{f}$ extends to a torsor $Y \to \Xs$ representing the class $\partial \A$.  One can verify that $Y(\F_{17})$ is empty, showing that the invariant of $\A$ is never zero on $X(\Q_{17})$; so the invariant map is constant with value $\tfrac{1}{2}$.  Indeed, all twelve points of $\Xs(\F_{17})$ lift to points on the quadratic twist of $Y$.  However, the Hasse--Weil bounds show that $Y$ does admit points over every non-trivial extension of $\F_{17}$.  It follows that $X$ is a variety with no $\Q$-rational points, but no Brauer--Manin obstruction to the existence of a rational zero-cycle of degree one.  Unfortunately, I have been unable to find a rational zero-cycle of degree one on $X$.
\end{remark}

As a final corollary to Theorem~\ref{thm:weil}, we give an alternative proof of the following result of Colliot-Th\'el\`ene and Saito~\cite{CTS:IMRN-1996}.
They used the Chebotarev density theorem in their proof; in ours, this is replaced by the appeal to the Weil conjectures in the proof of Theorem~\ref{thm:weil}.

\begin{corollary}
Let $\A$ be an element of $\Br X\nonp$.  The following statements are equivalent:
\begin{enumerate}[(i)]
\item\label{kerti} $\A \in (\Br \X + \Br k) \subset \Br X$;
\item\label{constZ} There exists $\alpha \in \Br k$ such that $(\A-\alpha)(z)=0$ for all $z \in Z_0(X,\X)$;
\item\label{zeroZ0} $\A(z)=0$ for all $z \in Z_0^0(X,\X)$;
\item\label{constZ1} $\A(z)$ is constant for $z \in Z_0(X,\X)$ of degree $1$.
\end{enumerate}
In particular, the left kernel of the evaluation pairing 
\[
(\Br X / \Br k) \times Z_0^0(X,\X) \to \Br k
\]
consists of the image of $\Br \X$ in $\Br
X/\Br k$.
\end{corollary}

\begin{proof}
The implication (\ref{kerti}) $\Rightarrow$ (\ref{constZ}) follows immediately from the fact that, if $\A$ lies in $\Br \X$, then we have $\A(z)=0$ for all $z \in Z_0(X,\X)$.  The implications (\ref{constZ}) $\Rightarrow$ (\ref{zeroZ0}) is trivial.  Given that, by assumption, $\X(\O)$ is non-empty, the equivalence of (\ref{zeroZ0}) and (\ref{constZ1}) is also straightforward.

It remains to prove that (\ref{zeroZ0}) implies (\ref{kerti}).  Suppose
that $\A \in \Br X$ does not satisfy~(\ref{kerti}); by
Proposition~\ref{prop:kert}, we have $\t(\A) \neq 0$.
Then Corollary~\ref{cor:z0} shows the $\A(z)$ takes several different values for $z \in Z_0^0(X,\X)$,
contradicting statement~(\ref{zeroZ0}).
\end{proof}

\subsection{Generalisation to several algebras}\label{sec:surj2}

A natural generalisation of Theorem~\ref{thm:weil} to several elements of $\Br X$ would say the following: given a collection of ``independent'' elements of $\Br X$, they should ``independently'' take all possible values when evaluated at points of $X(k)$, assuming that the residue field is sufficiently large.  In this section we make this statement precise and prove it.

\begin{definition}\label{def:li}
Let $A$ be a torsion Abelian group, and let $a_1, \dotsc, a_r$ be elements of $A$ with orders $n_1, \dotsc, n_r$ respectively.  We will say that the $a_i$ are \emph{linearly independent} if any of the following, clearly equivalent, conditions hold:
\begin{enumerate}[(i)]
\item The size of the subgroup generated by the $a_i$ is $\prod_i n_i$;
\item The homomorphism $\prod_i (\Z/n_i) \to A$, defined on the $i$th factor by $1 \mapsto a_i$, is injective;
\item Whenever $\sum_i \lambda_i a_i=0$ holds with $\lambda_i \in \Z$, then we have $n_i \mid \lambda_i$ for all $i$.
\end{enumerate}
\end{definition}

\begin{lemma}\label{lem:indep}
Let $G$ be a group, and let $\chi_1, \dotsc, \chi_r \in \Hom(G, \Q/\Z)$ be characters of $G$ with orders $n_1, \dotsc, n_r$ respectively.  Then the $\chi_i$ are linearly independent if and only if the product homomorphism
\[
\prod_i \chi_i \colon G \to \prod_i \big( \tfrac{1}{n_i}\Z \big) / \Z
\]
is surjective.
\end{lemma}
\begin{proof}
Replacing $G$ by its image under $\prod_i \chi_i$, we may assume that $G$ is finite and Abelian.
Denote by $\phi$ the homomorphism $\prod_i (\Z/n_i) \to \Hom(G, \Q/\Z)$ defined on the $i$th factor by $1 \mapsto \chi_i$.  As remarked in Definition~\ref{def:li}, the $\chi_i$ are linearly independent if and only if $\phi$ is injective.  But the homomorphism $\prod_i \chi_i$ is the Pontryagin dual of $\phi$, so is surjective if and only if $\phi$ is injective.
\end{proof}

%
\begin{theorem}\label{thm:indep}
Let $c_1, \dotsc, c_r$ be linearly independent classes in
$\H^1(\Xsbns, \Q/\Z)$ of orders $n_1, \dotsc, n_r$ respectively, all coprime to $p$.
Then there is some constant $B$, depending only on the $c_i$,
such that, whenever $|\F|>B$ and for all $r$-tuples of normalised
elements $\A_i, \dotsc, \A_r \in \Br X$ with $\t\A_i=c_i$, the
product of the evaluation maps
\[
\prod \A_i \colon \X(\O) \to \prod_i \Br k[n_i]
\]
is surjective.
\end{theorem}
\begin{proof}
Since the $\A_i$ are normalised, Lemma~\ref{lem:norm} shows that each
$\partial \A_i$ is of order $n_i$ in $\H^1(\Xsns, \Q/\Z)$.
For each $i = 1, \dotsc, r$, let $Y_i \to \Xsns$ be a torsor under
$\Z/n_i$ representing the class of $\partial(\A_i)$ in $\H^1(\Xsns,
\Z/n_i)$.  Let $Y$ be the fibre product
\[
Y = Y_1 \times_{\Xsns} \dotsb \times_{\Xsns} Y_r.
\]
Let $\bar{Y}$ be the base change of $Y$ to $\Fb$.  
Then $\bar{Y}$ represents the class $\prod_i
c_i$ in $\H^1(\Xsbns, \prod_i \Z/n_i)$.  Now, identifying that
cohomology group with $\Hom(\pi_1(\Xsbns), \prod_i \Z/n_i)$,
Lemma~\ref{lem:indep} shows that the corresponding homomorphism is
surjective; by Lemma~\ref{lem:conn}, $\bar{Y}$ is connected.  Then
Fact~\ref{weil} shows that whenever $|\F| > B(\bar{Y})$, any variety
over $\F$ geometrically isomorphic to $\bar{Y}$ admits an $\F$-point.

Now let $(\alpha_1, \dotsc, \alpha_r)$ be an element of $\prod_i \Br
k[n_i]$; applying the residue map gives $(\beta_1, \dotsc, \beta_r)
\in \prod_i \H^1(\F, \Z/n_i)$.  Thinking of $\beta = \prod_i \beta_i$
as a class in $\H^1(\F, \prod_i \Z/n_i)$, we have the Galois twist
\[
Y^\beta = Y_1^{\beta_1} \times_{\Xsns} \dotsb \times_{\Xsns} Y_r^{\beta_r}.
\]
So $Y^\beta(\F) \neq \emptyset$ if and only if $Y_i^{\beta_i}(\F) \neq
\emptyset$ for all $i$, which by Lemma~\ref{lem:torsors} happens if
and only if there is a point $P \in \X(\O)$ satisfying
$\A_i(P)=-\alpha_i$ for all $i$.  Because $Y^\beta$ is geometrically
isomorphic to $\bar{Y}$, this will happen if $|\F|>B$ irrespective of
the choice of $(\alpha_1, \dotsc, \alpha_r)$.
\end{proof}

\section{Application: reduction to a cone}\label{sec:cone}

In this section, we look at the situation where $\Xs$ is the
projective cone over a smooth variety.  The purpose is to demonstrate
how the condition $\Br \Xsb=0$ allows us to deduce surjectivity of the
evaluation maps for any linearly independent set of elements of $\Br X
/ \Br k$; under additional hypotheses, we can go further and prove
vanishing of the Brauer--Manin obstruction over a global field.  Let
us first gather some geometric information about cones.

\begin{proposition}\label{prop:cone}
Let $K$ be an algebraically closed field, and let $Y$ be the projective cone in $\mathbb{P}_K^d$ over a smooth subvariety $Z$ of $\mathbb{P}_K^{d-1}$.  Suppose that $Y$ is normal.  Then:
\begin{enumerate}[(i)]
\item\label{conepic} the Picard group of $Y$ is isomorphic to $\Z$, and is generated by the class of a hyperplane section;
\item\label{conebr} the prime-to-$p$ torsion in the Brauer group of $Y$ is trivial, where $p$ is the characteristic of $K$.
\end{enumerate}
\end{proposition}
\begin{proof}
Let $U \subset Y$ be the affine cone over $Z$.
Hoobler~\cite{Hoobler:MRT-84} studied graded rings $S$ and gave a list of functors $F$ which satisfy $F(S) = F(S_0)$, where $S_0$ is the degree-zero part of $S$.  In particular, he showed that $\Pic S  = \Pic S_0$ when $S$ is normal, and that $\Br S \nonp = \Br S_0 \nonp$.  We apply this with $S$ being the affine coordinate ring of $U$ and $S_0=K$ to show that $\Pic U$ and $\Br U \nonp$ are both trivial.  Because $U$ is the complement of a hyperplane section in $Y$, we obtain~(\ref{conepic}).

To complete the calculation of $\Br Y$, consider the open set $V \subset Y$ obtained by removing the vertex.  The Mayer--Vietoris sequence (see~\cite[III.2.24]{Milne:EC}) for the covering $Y = U \cup V$ contains the exact sequence
\[
\Pic U \oplus \Pic V \to \Pic(U \cap V) \to \Br Y \to \Br U \oplus \Br V \to \Br(U \cap V).
\]
Because $V$ is smooth, $\Pic V$ surjects onto $\Pic(U \cap V)$ and $\Br V$ injects into $\Br(U \cap V)$.  We deduce $\Br Y \nonp=0$, proving~(\ref{conebr}).
\end{proof}

\begin{remark}
In the case of a cone over a smooth projective \emph{curve} in characteristic $p$, Gordon~\cite{Gordon:JA-82} showed that $\Br Y$ is isomorphic to the set of $K$-points of a unipotent algebraic group.  In particular, $\Br Y$ is $p$-torsion.
\end{remark}

We obtain the following consequence.

\begin{lemma}\label{lem:cone}
Let $\X$ be a regular scheme, proper and flat over $\O$, such that the
generic fibre $X$ is a smooth variety over $k$ and the special fibre
$\Xs$ is the projective cone over a smooth projective variety.  Then
the kernel of the map $\bar\partial \colon \Br X \nonp \to
\H^1(\Xsbns, \Q/\Z)$ is $\Br k \nonp$.
\end{lemma}
\begin{proof}
We first show that $\Br \Xs \nonp$ is trivial, as follows.  The
Hochschild--Serre spectral sequence for $\Xsb \to \Xs$ gives an exact
sequence
\[
\Br \F \to \ker( \Br\Xs \to \Br\Xsb ) \to \H^1(\F, \Pic \Xsb).
\]
Since $\F$ is finite, $\Br\F$ is trivial; and, since $\Pic\Xsb$ is
isomorphic to $\Z$ with trivial Galois action, $\H^1(\F, \Pic\Xsb)$ is trivial. Therefore
$\Br\Xs$ injects into $\Br\Xsb$, and so $\Br\Xs \nonp$ is
trivial. The result now follows from Proposition~\ref{prop:proper} and a diagram-chase using Lemma~\ref{lem:nr}.
\end{proof}

\begin{corollary}
Under the conditions of Lemma~\ref{lem:cone}, suppose that $\A_1,
\dotsc, \A_r \in \Br X \nonp$ are linearly independent when considered
as elements of $(\Br X / \Br k)$.  Then the conclusion of
Theorem~\ref{thm:indep} applies to $\A_1, \dotsc, \A_r$. 
\end{corollary}

In particular situations, we can make the bounds of
Theorem~\ref{thm:indep} explicit.  One such case is when $\Xs$ is the
projective cone over a smooth projective curve $C$ of genus $g$.

\begin{theorem}\label{thm:cone}
Let $\X$ be a regular scheme, proper and flat over $\O$, such that the
generic fibre $X$ is a smooth surface over $k$ and the special fibre
$\Xs$ is the projective cone over a smooth projective curve of genus
$g$.  Let $\A_1, \dotsc, \A_r$ be normalised elements of $\Br X$, of
orders coprime to $p$, the
images of which in $\Br X / \Br k$ are linearly independent of orders
$n_1, \dotsc, n_r$, and let $N=n_1 \dotsm n_r$.  Suppose that the
residue field $\F$ satisfies
\begin{equation}\label{eq:size}
|\F| > \big(g' + \sqrt{(g')^2-1}\big)^2, \quad \text{where} \quad
g' = N(g-1)+1.
\end{equation}
Then the product of the evaluation maps
\[
\prod \A_i \colon X(k) \to \prod_i \Br k[n_i]
\]
is surjective.
\end{theorem}

\begin{proof}
As in the proof of Theorem~\ref{thm:indep}, it will be enough to show
that any \'etale cover of $\Xsns$ of degree $N$ admits an
$\F$-rational point.  Suppose that $\Xs$ is the projective cone over a
smooth projective curve $C$ of genus $g$.  After removing the vertex,
$\Xsns$ is a fibration over $C$, with fibres isomorphic to the affine
line, and every \'etale cover of $\Xsns$ arises by pulling back an
\'etale cover $D \to C$.  If $D \to C$ is of degree $N$, then the
Riemann--Hurwitz formula shows that the genus of $D$ is $g'$.  Now the
Hasse--Weil bounds show that, under the condition~\eqref{eq:size}, every
smooth projective curve of genus $g'$ over $\F$ admits an
$\F$-rational point, and hence every \'etale cover of $\Xsns$ of
degree $N$ admits an $\F$-rational point.
\end{proof}

\begin{remark}
If $g = 1$, then the condition~\eqref{eq:size} is vacuous.
\end{remark}

\section{Vanishing of Brauer--Manin obstructions}\label{sec:global}

An application of the surjectivity results of Sections~\ref{sec:surj} and~\ref{sec:surj2} is to show that certain varieties over global fields have no Brauer--Manin obstruction to the existence of a rational point.

Let us first make a useful definition.  Let $X$ be a variety over a number field $K$, and let $v$ be a place of $K$.  Suppose that $X(K_v)$ is non-empty, and let $P$ be a point of $X(K_v)$.  Using $P$, we can define a pairing
\[
\Br X / \Br K \times X(K_v) \to \Q/\Z
\]
by
\[
(\A, Q) \mapsto \inv_v \A(Q) - \inv_v \A(P).
\]
Indeed, for any $Q$, this defines a homomorphism $\Br X \to \Q/\Z$ which is trivial on $\Br K$, so factors through $\Br X / \Br K$.

\begin{definition}
Let $B$ be a subgroup of $\Br X / \Br K$.  We say that $B$ is \emph{prolific} at the place $v$ if the map 
\[
X(K_v) \to \Hom(B,\Q/\Z)
\]
induced by the above pairing is surjective.
\end{definition}

\begin{remark}
\begin{enumerate}
\item The definition does not depend on the choice of $P$, since changing $P$ simply translates the image by an element of $\Hom(B,\Q/\Z)$.
\item Suppose that $B$ is generated by a single algebra $\A$ of order $n$, which we may assume to be normalised.  Then $B$ is prolific if and only if the evaluation map $\A \colon X(K_v) \to \Br k[n]$ is surjective.
\item In order for $B$ to be prolific, the localisation map $B \to (\Br X_{K_v} / \Br K_v)$ must be injective.
\end{enumerate}
\end{remark}

The motivation for making this definition is the following easy proposition.

\begin{proposition}\label{prop:prolific}
Let $X$ be a variety over a number field $K$.  Let $B$ be a subgroup of $\Br X / \Br K$, and suppose that $B$ is prolific at a place $v$ of $K$.  Then there is no Brauer--Manin obstruction to the existence of rational points on $X$ coming from $B$.
\end{proposition}
\begin{proof}
Let $\Omega$ denote the set of places of $K$; to prove the proposition, we must find an adelic point
$(P_w \in X(K_w))_{w \in \Omega}$ satisfying $\sum_{w \in \Omega} \inv_w
\A(P_w)=0$ for all $\A \in B$.  Let $(P_w)_{w \in \Omega}$ be \emph{any}
adelic point of $X$.  Define a homomorphism $\phi \colon B \to \Q/\Z$ by
\[
\phi(\A) = \sum_{w \in \Omega} \inv_w \A(P_w).
\]
Since $B$ is prolific at $v$, there is a point $Q \in X(K_v)$ such that
\[
-\phi(\A) = \inv_v \A(Q) - \inv_v \A(P_v) \text{ for all } \A \in B.
\]
Replacing $P_v$ by $Q$ then gives an adelic point satisfying $\sum_{w \in \Omega} \inv_w \A(P_w)=0$ for all $\A \in B$, and so there is no Brauer--Manin obstruction to the existence of a rational point on $X$.
\end{proof}

Let us now use the calculations of Section~\ref{sec:cone} to show
vanishing of the Brauer--Manin obstruction on some varieties reducing
to a cone at a bad prime.

\begin{theorem}\label{thm:cone-global}
Let $X$ be a smooth, projective, geometrically integral surface over a
number field $K$, with points in every completion of $K$.  Let
$\bar{X}$ denote the base change of $X$ to an algebraic closure
$\bar{K}$ of $K$.  Suppose that $\H^1(X, \O_X)$ is trivial, that $\Pic
\bar{X}$ is torsion-free and that $\Br X / \Br K$ is finite of order
$N$.  Suppose that there is a prime $\p$ of $K$ such that
\begin{enumerate}[(i)]
\item $\p$ does not divide $N$;
\item $X$ extends to a regular scheme, projective and flat over the
  ring of integers at $\p$, the special fibre of which is the
  projective cone over a smooth projective curve of genus $g$; and
\item the residue field at $\p$ satisfies condition~\eqref{eq:size} of Theorem~\ref{thm:cone}.
\end{enumerate}
Then $\Br X / \Br K$ is prolific at $\p$, and so there is no Brauer--Manin obstruction to the existence of rational points on $X$.
\end{theorem}

Before proving Theorem~\ref{thm:cone-global}, we first prove a series of general lemmas.

\begin{lemma}\label{lem:H1inj}
Let $G$ be a finite group and $H \subseteq G$ a subgroup.  Let $M$ be a
$G$-module which is torsion-free as a $\Z$-module, and suppose that $M^H$
is the same as $M^G$.  Then the restriction map $\H^1(G,M) \to
\H^1(H,M)$ is injective.
\end{lemma}

\begin{proof}
Let us first notice that $(M \otimes \Q)^H$ is the same as $(M \otimes
\Q)^G$; this follows from the observation that $(M \otimes \Q)^G = M^G
\otimes \Q$, and similarly for $H$, and the fact that taking the
tensor product with $\Q$ preserves surjectivity of maps.  Now consider
the short exact sequence of $G$-modules
\[
0 \to M \to M\otimes \Q \to M \otimes (\Q/\Z) \to 0.
\]
Taking cohomology of both $G$ and $H$ results in the following
diagram, which is commutative with exact rows.
\[
\begin{CD}
(M \otimes \Q)^G @>>> (M \otimes \Q/\Z)^G @>>> \H^1(G,M) @>>> 0 \\
@VVV @VVV @V{\res}VV @VVV \\
(M \otimes \Q)^H @>>> (M \otimes \Q/\Z)^H @>>> \H^1(H,M) @>>> 0 \\
\end{CD}
\]
The second vertical map is an inclusion, so is injective.  The
injectivity of the restriction map on $\H^1$ now follows from the Five
Lemma.
\end{proof}

\begin{remark}
If $H$ is a normal subgroup of $G$, then the lemma follows immediately
from the inflation-restriction exact sequence.
\end{remark}

\begin{lemma}\label{lem:brdiv}
Let $X$ be a variety over a number field $K$, and suppose that $X$ has
points in each real completion of $K$.  Then, for any integer $n$, the
natural map $\Br X[n] \to (\Br X/\Br K)[n]$ is surjective.
\end{lemma}
\begin{proof}
The fundamental exact sequence of class field theory
\[
0 \to \Br K \to \bigoplus_v \Br K_v \xrightarrow{\sum_v \inv_v} \Q/\Z
\to 0,
\]
being split, remains exact upon taking the tensor product with $\Zn$;
we see that the induced map $(\Br K)/n \to \bigoplus_v (\Br K_v)/n$ is
an isomorphism.  Therefore an element of $\Br K$ is divisible by $n$
if and only if it is divisible by $n$ in each $\Br K_v$.  For each
finite place $v$, we know that $\Br K_v \cong \Q/\Z$ is divisible; the
same is trivially true at complex places.  So whether an element of
$\Br K$ is divisible by $n$ is determined at the real places.

Let $\A \in \Br X$ be such that $n\A$ is equivalent to a constant
class $\alpha \in \Br K$.  We must show that we can change $\A$ by a constant class to obtain an element of $\Br X[n]$.  We have local points $P_v \in X(K_v)$ for
each real place $v$.  It follows that, for each real place $v$, the
image of $\alpha$ in $\Br K_v$ is divisible by $n$, for $\alpha(P_v) =
n \A(P_v)$.  So $\alpha \in \Br K$, being divisible by $n$ in each
$\Br K_v$, is also divisible by $n$ in $\Br K$.  Let $\beta \in \Br K$
satisfy $n\beta = \alpha$.  Then we have $n(\A-\beta)=0$, that is,
$\A-\beta$ lies in $\Br X[n]$.
\end{proof}

\begin{lemma}\label{lem:algclosed}
Let $L \subset M$ be an extension of algebraically closed fields of characteristic zero, and let $X$ be a smooth, proper variety over $L$.  Denote by $X_M$ the base change of $X$ to $M$.  Then the natural maps $\NS(X) \to \NS(X_M)$ and $\Br X \to \Br X_M$ are isomorphisms.
\end{lemma}
\begin{proof}
We can interpret $\NS(X)$ as the group of connected components of the Picard scheme of $X$.  Since $L$ is algebraically closed, the base change $M/L$ induces a bijection on connected components; so
the map on N\'eron--Severi groups $\NS(X) \to \NS(X_M)$ is an isomorphism.  Now, for each integer $n$, the Kummer sequence gives a commutative diagram as follows:
\[
\begin{CD}
0 @>>> \NS(X)/n @>>> \H^2(X, \mmu_n) @>>> \Br X[n] \to 0 \\
@. @VVV @VVV @VVV \\
0 @>>> \NS(X_M)/n @>>> \H^2(X_M, \mmu_n) @>>> \Br X_M[n] \to 0.
\end{CD}
\]
The middle vertical arrow is an isomorphism by~\cite[VI.2.6]{Milne:EC}, and so the right-hand arrow is also an isomorphism.  Because $\Br X$ and $\Br X_M$ are torsion groups, the result follows. 
\end{proof}

\begin{lemma}\label{lem:brinj}
Let $X$ be a smooth, geometrically irreducible, projective variety over a global field $K$.  Let $v$ be a place of $K$ and write $X_v$ for the base change of $X$ to $K_v$.  Suppose that $X(K_v)$ is non-empty, and that $\Pic X_v$ is generated by the class of a hyperplane section.  Then the natural map $(\Br X / \Br K) \to(\Br X_v / \Br K_v)$ is injective.
\end{lemma}
\begin{proof}
Let $\bar{K}$ be an algebraic closure of $K$, and denote by $\bar{X}$ the base change of $X$ to $\bar{K}$.
From the
Hochschild--Serre spectral sequence for $\bar{X} \to X$ we deduce the
following standard exact sequence:
\[
0 \to \ker(\H^1(K, \Pic\bar{X}) \to \H^3(K,\bar{K}^\times)) \to \Br X / \Br K \to \Br \bar{X}.
\]
Let $\bar{K}_v$ be an algebraic closure of $K_v$ containing
$\bar{K}$.  Writing $\bar{X}_v$ for the base change of $X$ to
$\bar{K}_v$, we obtain a similar exact sequence for $X_v$.  Since $\H^3(K,\bar{K}^\times)$ and $\H^3(K_\p, \bar{K}_v^\times)$ are both trivial, it will be sufficient to show that the two maps
\[
\H^1(K, \Pic\bar{X}) \to H^1(K_v, \Pic\bar{X}_v) \quad \text{and}
\quad \Br \bar{X} \to \Br \bar{X}_v
\]
are both injective; the result will then follow by the Five Lemma.  The map $\Br \bar{X} \to \Br \bar{X}_v$ is actually an isomorphism, by Lemma~\ref{lem:algclosed}; so it remains to show that $\H^1(K, \Pic\bar{X}) \to H^1(K_v, \Pic\bar{X}_v)$ is injective.

Since $\Pic X_v$ is isomorphic to $\Z$, it follows that the Abelian variety $A=\Picsh^0 X$ is trivial: otherwise, $A(K_v)$ would be uncountable, and because $X$ admits a $K_v$-point we have $A(K_v) = \Pic^0(X_v)$, so that $\Pic X_v$ would also be uncountable.  Therefore $\Pic X$ is the same as the N\'eron--Severi group $\NS(X)$, and similarly for $\bar{X}$ and $\bar{X}_v$.  

By Lemma~\ref{lem:algclosed}, the natural map
$\Pic \bar{X} \to \Pic \bar{X}_v$ is an isomorphism.  Take $L \subset \bar{K}$ to be a finite
Galois extension over which a set of generators for $\Pic\bar{X}$ is
defined, so that $\Pic X_L \to \Pic \bar{X}$ is also an isomorphism.  Let $L_v$ be the completion of $L$ in $\bar{K}_v$; then Hilbert's Theorem~90 shows that $\Pic X_{L_v} \to \Pic \bar{X}_v$ is injective, and therefore also an isomorphism.
Let $G$ denote the Galois group of $L/K$.  The inclusion $L \subset L_v$ identifies $\Gal(L_v/K_v)$ with the decomposition group $G_v \subset G$, and the two inflation-restriction sequences form a commutative diagram
\[
\begin{CD}
0 @>>> \H^1(G, \Pic X_L) @>>> \H^1(K, \Pic \bar{X}) @>>> \H^1(L, \Pic \bar{X}) \\
@. @VVV @VVV @VVV \\
0 @>>> \H^1(G_v, \Pic X_{L_v}) @>>> \H^1(K_v, \Pic \bar{X}_v) @>>> \H^1(L_v, \Pic \bar{X}_v) \\
\end{CD}
\]
Because $\Pic\bar{X}$ is a finitely generated free $\Z$-module on which the absolute Galois group of $L$ acts trivially, the group $\H^1(L, \Pic\bar{X})$ vanishes, as does $\H^1(L_v, \Pic\bar{X}_v)$. 
So the map
$\H^1(K, \Pic\bar{X}) \to H^1(K_v, \Pic\bar{X}_v)$ can be identified
with the restriction map $\H^1(G,\Pic X_L) \to
\H^1(G_v,\Pic X_L)$.  Because $X$ admits a $K_v$-point, we have $(\Pic X_{L_v})^{G_v} \isom \Pic X_v$, and so $(\Pic X_{L_v})^{G_v}$ is generated by the class of a hyperplane section.  This class is defined over $K$, and so $(\Pic X_L)^G$ and $(\Pic X_{L_v})^{G_v}$ coincide; by Lemma~\ref{lem:H1inj} we deduce that $\H^1(K, \Pic\bar{X}) \to H^1(K_v, \Pic\bar{X}_v)$ is an injection, completing the proof.
\end{proof}

\begin{proof}[of Theorem~\ref{thm:cone-global}]
Let $\X$ be the given model over the ring of integers of $K_\p$, and let $X_0$ be the special fibre.  The restriction map $\Pic \X \to \Pic X_\p$ is surjective because $\X$ is regular, and its kernel is generated by the class of the special fibre, which is principal; so we have $\Pic\X \cong \Pic X_\p$.  By
Proposition~\ref{prop:cone}, $\Pic \bar{X}_0$ is generated by the class of a
hyperplane section.  It follows that $\H^1(X_0, O_{X_0})=0$.
By~\cite[Corollaire~3 to Th\'eor\`eme~7]{Grothendieck:GFGA}, the map $\Pic \X \to \Pic X_0$ is injective, and so
$\Pic X_\p$ is also generated by the class of a hyperplane section.
Now Lemma~\ref{lem:brinj} shows that the natural map $(\Br X / \Br K) \to (\Br X_\p / \Br K_\p)$ is injective.  

Let $\A_1, \dotsc, \A_r$ be a minimal
set of generators for $\Br X / \Br K$.  Then $\A_1, \dotsc, \A_r$ are also
linearly independent as elements of $\Br X_\p / \Br K_\p$.  By Lemma~\ref{lem:brdiv}, we
can choose the $\A_i$ such that each has the same order $n_i$ in $\Br
X$ as it does in $\Br X / \Br K$.  Further changing
each $\A_i$ by an element of $\Br K[n_i]$, we may assume that they are
normalised.  Applying Theorem~\ref{thm:cone}, we find that the product
of the evaluation maps
\[
\prod \A_i \colon X(K_\p) \to \prod_i \Br K_\p[n_i]
\]
is surjective, and therefore $\Br X / \Br K$ is prolific.  By Proposition~\ref{prop:prolific}, there is no Brauer--Manin obstruction to the existence of rational points on $X$.
\end{proof}

As special cases of Theorem~\ref{thm:cone-global}, we recover the
following results from the literature.

\begin{enumerate}
\item (Colliot-Th\'el\`ene, Kanevsky, Sansuc \cite{CTKS})
Let $X \subset \P^3_K$ be a smooth diagonal cubic surface,
\[
X \colon a_0 X_0^3 + a_1 X_1^3 + a_2 X_2^3 + a_3 X_3^3 = 0,
\] 
with $a_i$ coprime integers, having points over every completion of $K$.
Suppose that $p$ is a prime such that $v_p(a_0a_1a_2a_3)=1$.  Then
there is no Brauer--Manin obstruction to the existence of rational
points on $X$.

\item (Bright \cite{Bright:cubics}) Let $X \subset \P^3_K$ be a smooth
  cubic surface defined by an equation of the form
\[
f(X_0,X_1,X_2) + p g(X_0,X_1,X_2,X_3) = 0
\]
with $f,g$ cubic forms with integer coefficients and $p \nmid
g(0,0,0,1)$, having points over every completion of $K$.  (The condition
$p \nmid g(0,0,0,1)$ ensures that the given model is regular at $p$,
and conversely any cubic surface, regular at $p$, and having reduction
modulo $p$ a cone may be put into this form.)  Then there is no
Brauer--Manin obstruction to the existence of rational points on $X$.

\item (Bright \cite{Bright:vanishing})  Let $X \subset \P^3_K$ be a
  smooth diagonal quartic surface,
\[
X \colon a_0 X_0^4 + a_1 X_1^4 + a_2 X_2^4 + a_3 X_3^4 = 0,
\] 
with $a_i$ coprime integers, having points over every completion of $K$.
Let $H$ be the subgroup of $\Q^\times / (\Q^\times)^4$ generated by
$-1$, $4$ and the quotients $a_i / a_j$; suppose that $H$ has order
$256$ (the most general case) and that $H \cap \{ 2,3,5 \} =
\emptyset$.  (These conditions ensure that $\Br X$ has order $2$.)
Suppose that there exists a prime $p$ such that $p>100$ and
$v_p(a_0a_1a_2a_3)=1$.  Then there is no Brauer--Manin obstruction to
the existence of rational points on $X$.
\end{enumerate}

Similar applications should exist whenever a bound on the order of the Brauer group is known.  That is in general a difficult problem, but the case of diagonal quartic surfaces is well understood thanks to work of Ieronymou, Skorobogatov and Zarhin.  To conclude, we deduce from Theorem~\ref{thm:cone-global} a general result for diagonal quartics.

\begin{corollary}
Let $X \subset \P^3_\Q$ be the smooth diagonal quartic surface defined by the equation
\[
a_0 X_0^4 + a_1 X_1^4 + a_2 X_2^4 + a_3 X_3^4 = 0
\]
with $a_i$ coprime integers, and suppose that $X$ has points over every completion of $\Q$.  Suppose that there exists a prime $p > (2^{54}+2^{28}+1)$ satisfying $v_p(a_0 a_1 a_2 a_3) = 1$.  Then there is no Brauer--Manin obstruction to the existence of rational points on $X$.
\end{corollary}
\begin{proof}
Because $X$ is a K3 surface, the condition $\H^1(X,\O_X)=0$ is satisfied and $\Pic \bar{X}$ is torsion-free (in fact, free of rank $20$).  The condition $v_p(a_0 a_1 a_2 a_3)=1$ ensures that the given equation defines a model for $X$ over $\Z_p$ that is regular; the special fibre is the projective cone over a smooth diagonal quartic curve, of genus $3$.
Ieronymou and Skorobogatov have recently shown~\cite[Theorem~1.1]{IS:odd} that this condition also implies the absence of any odd torsion in $\Br X / \Br \Q$; and the same authors with Zarhin showed in~\cite{ISZ:JLMS-2011} that the $2$-torsion in $\Br X / \Br \Q$ has order at most $2^{25}$.  So, for any $p$ larger than the claimed bound, Theorem~\ref{thm:cone-global} applies and shows that there is no Brauer--Manin obstruction.
\end{proof}

It should be noted that Ieronymou and Skorobogatov~\cite{IS:odd} have shown that odd-order elements in the Brauer groups of diagonal quartic surfaces over $\Q$, when they do exist, can never obstruct the Hasse principle.  They do this by calculating that, if $\A$ has order $p=3$ or $5$ in $\Br X / \Br \Q$, then the corresponding evaluation map $\A \colon X(\Q_p) \to \Br \Q_p[p]$ is surjective.  Because this involves studying elements of order $p$ in the Brauer group at the prime $p$, the methods of this paper are useless in understanding that result.

\paragraph{acknowledgements}
Substantial parts of this work were carried out during the semester on `Rational points and algebraic cycles' at the Centre Interfacultaire Bernoulli of the \'Ecole Polytechnique F\'ed\'erale de Lausanne in 2012; I thank the organisers and staff for their support and hospitality.  I also thank the Center for Advanced Mathematical Sciences of the American University of Beirut for their support.

\bibliographystyle{abbrv}
\bibliography{martin}

\end{document}